\documentclass[a4paper,reqno,12pt]{amsart}

\usepackage{amssymb,amsthm,amsmath}
\usepackage{xcolor}
\usepackage{hyperref}
\hypersetup{
	colorlinks=true,
	linkcolor=blue,
	citecolor=red,
}
\usepackage[left=3cm,right=2.7cm,top=3cm,bottom=2.7cm,twoside]{geometry}

\newtheorem{defin}{Definition}[section]
\newenvironment{definition}{\begin{defin}\rm}{\end{defin}}
\newtheorem{theorem}[defin]{Theorem}
\newtheorem{exa}[defin]{Example}
\newenvironment{example}{\begin{exa}\rm}{\end{exa}}\newtheorem{lemma}[defin]{Lemma}
\newtheorem{corollary}[defin]{Corollary}
\newtheorem{question}{Question}

\newenvironment{remark}
{\par\vspace{2.5mm}\noindent{\bf Remark.}}{\par\vspace{2.5mm}}

\newcommand{\N}{\mathbb{N}}

\newcommand{\D}{\mathcal{D}}

\numberwithin{equation}{section}

\makeatletter
\@namedef{subjclassname@2020}{%
	\textup{2020} Mathematics Subject Classification}
\makeatother

\begin{document}

\title[Asymptotic $\psi$-density of subsets of positive integers]{Asymptotic $\psi$-density of subsets of\\ positive integers}

\author{J.~Heittokangas and Z.~Latreuch}

\begin{abstract}
The sizes of subsets of the natural numbers are typically quantified in terms of asymptotic (linear) and logarithmic densities. These concepts have been generalized to weighted $w$-densities, where a specific weight function $w$ plays a key role. In this paper, a parallel theory of asymptotic $\psi$-densities is introduced, where the weight is expressed slightly differently in terms of   differentiable functions $\psi$, which are either concave or convex and satisfy certain asymptotic properties. Alternative new proofs for known results on analytic and Abel densities are also given.\\[5pt]
\end{abstract}

\maketitle

\renewcommand{\thefootnote}{}
\footnotetext[1]{Corresponding author: J.~Heittokangas.}
\footnotetext[2]{Z.~Latreuch was supported by the University of Mostaganem and the National Higher School of Mathematics. He would like to thank the Department of Physics and Mathematics at the UEF for its hospitality during his visit in March 2023.}


\section{Introduction}

\thispagestyle{empty}
The well-known asymptotic and logarithmic densities of a subset $A$ of the set $\N$ of natural numbers are given, respectively, by
	$$
	\operatorname{d}(A)=\lim_{n\to\infty}\frac{\sum_{k=1}^n\chi_A(k)}{n}
	\quad\text{and}\quad
	\operatorname{\delta}(A)=\lim_{n\to\infty}\frac{\sum_{k=1}^n\frac{\chi_A(k)}{k}}{\log n}
	$$
provided that the limits exist. Here $\chi_A$ denotes the characteristic function of the set $A$. Replacing the limits with limit inferiors or with limit superiors, which always exist, we get the corresponding lower densities $\underline{\operatorname{d}}(A),\, \underline{\operatorname{\delta}}(A)$ and the upper densities $\overline{\operatorname{d}}(A),\, \overline{\operatorname{\delta}}(A)$. These quantities are related to each other by means of the well-known inequalities 
	\begin{equation}\label{asymptotic-logarithmic}
	0\leq \underline{\operatorname{d}}(A)
	\leq \underline{\operatorname{\delta}}(A)
	\leq \overline{\operatorname{\delta}}(A)
	\leq \overline{\operatorname{d}}(A)\leq 1,
	\end{equation}
see \cite[p.~241]{HR} or \cite[p.~416]{Tenenbaum}. In the literature, the asymptotic density is also known by the names of linear density and natural density.

Being able to quantify the sizes of sets $A\subset\N$ is particularly important in (analytic) number theory \cite{NZM, Tenenbaum} and in the theories of sequences \cite{HR} and series \cite{PS}. As for a simple example, if $A$ is the set of all even (or odd) natural numbers and if $B$ is the set of all prime numbers, then $\operatorname{d}(A)=1/2$ and $\operatorname{d}(B)=0$. The latter is a simple consequence of the celebrated prime number theorem. 

Naturally, the situation calls for a more advanced quantification of sets. The concept of a \textit{weighted density}, introduced in \cite{A, RV}, provides a generalization of both asymptotic and logarithmic densities. Indeed, a weighted density is determined by a function $w: \mathbb{N}\to (0,\infty)$, known as a \textit{weight function}, which satisfies the  conditions
	\begin{equation}\label{ulam}
	\sum_{k=1}^{\infty}w(k)=\infty 
	\quad \text{and}\quad 
	\lim_{n\to\infty}\frac{w(n)}{\sum_{k=1}^{n}w(k)}=0.
	\end{equation}
Given a set $A\subset \N$, the \textit{lower and upper weighted $w$-densities with respect to $w$} are defined, respectively, by
	\begin{align}\label{w-density}
	\underline{\operatorname{d}}_w(A)
	=\liminf_{n\to\infty}\frac{\sum_{k=1}^n w(k)\chi_A(k)}{\sum_{k=1}^nw(k)} 
	\quad \text{and}\quad
	\overline{\operatorname{d}}_w(A)
	=\limsup_{n\to\infty}\frac{\sum_{k=1}^n w(k)\chi_A(k)}{\sum_{k=1}^nw(k)},
	\end{align}
Furthermore, if $f$ and $g$ are two weight functions such that the ratio $f(n)/g(n)$ is a decreasing function, then the following inequalities hold for any set $A\subset \mathbb{N}$:
	\begin{equation}\label{inq1}
	0\leq \underline{\operatorname{d}}_g(A)
	\leq \underline{\operatorname{d}}_f(A)
	\leq \overline{\operatorname{d}}_f(A)
	\leq \overline{\operatorname{d}}_g(A)\leq 1.
	\end{equation}
These inequalities can be derived from the following Rajagopal's classical result.

\begin{theorem}[{\cite[Theorem 3]{R}}]\label{L1}
Let $(s_n)_n$ be a real sequence, and let $(a_n)_n$ and $(b_n)_n$ be positive real sequences such that
	$$
	A_n=\sum_{k=1}^na_k\to \infty,\quad B_n=\sum_{k=1}^nb_n\to \infty,\quad n\to \infty.
	$$
Denote
	$$
	\sigma(s_n,a_n)=\frac{\sum_{k=1}^na_ns_n}{A_n} 
	\quad\text{and}\quad
	\sigma(s_n,b_n)=\frac{\sum_{k=1}^nb_ns_n}{B_n}.
	$$
If $(a_n/b_n)_n$ is decreasing, then
	$$
	\liminf_{n\to\infty}\sigma(s_n,b_n) \leq \liminf_{n\to\infty}\sigma(s_n,a_n)\leq \limsup_{n\to\infty}\sigma(s_n,a_n)\leq\limsup_{n\to\infty}\sigma(s_n,b_n).
	$$
\end{theorem}

It is worth noting that when considering a real subset $E$ of the interval $\left(r_0, R\right)$, where $0 < r_0 < R\leq \infty$, instead of an integer subset $A\subset \mathbb{N}$, the lower and upper $\psi$-densities of $E\subset \left(r_0, R\right)$ are defined \cite{B}, respectively, by
	\begin{eqnarray}\label{Barry}
	\underline{\psi-\operatorname{dens}}(E)=\liminf _{r \rightarrow R^{-}} \frac{\int_{E \cap\left[r_0, r\right)} d \psi(t)}{\psi(r)} , 
	\quad \overline{\psi-\operatorname{dens}}(E)=\limsup _{r \rightarrow R^{-}} \frac{\int_{E \cap\left[r_0, r\right)} d \psi(t)}{\psi(r)}.
	\end{eqnarray}
In this context, $\psi$ is a positive, unbounded, differentiable, and strictly increasing function on $\left(r_0, R\right)$. 
Analogously to \eqref{inq1}, it has been shown in \cite[Lemma 1]{B} that
	\begin{eqnarray}\label{inq2}
	0\leq \underline{\psi-\operatorname{dens}}(E) 
	\leq \underline{\phi-\operatorname{dens}}(E) 
	\leq \overline{\phi-\operatorname{dens}}(E) 
	\leq \overline{\psi-\operatorname{dens}}(E)\leq 1,
	\end{eqnarray}
assuming that $\psi$ is convex with respect to $\varphi$. 
Note that when the limit inferiors and superiors coincide in \eqref{Barry}, the $\psi$-density of $E$ can be expressed as
	\begin{eqnarray}\label{new3}
	\psi-\operatorname{dens}(E)=\lim _{r \rightarrow R^{-}} \frac{\int_{E \cap\left[r_0, r\right)} d\psi(t)}{\psi(r)}
	=\lim _{r \rightarrow R^{-}}\frac{\int_{E \cap[r_0, r)} \psi^{\prime}(t)\, dt}{\int_{[r_0, r)} \psi^{\prime}(t)\, dt}.
	\end{eqnarray}
By considering $E \subset \mathbb{N}$ and $R=\infty$ in \eqref{new3}, and substituting Lebesgue measure with the \textit{counting measure}, we obtain 
	\begin{eqnarray}\label{newdef}
	\psi-\operatorname{dens}(E)=\lim _{n \rightarrow \infty} \frac{\sum_{k \in E \cap \mathbb{N}_n} \psi^{\prime}(k)}{\sum_{k=1}^n \psi^{\prime}(k)}
	=\lim _{n \rightarrow \infty} \frac{\sum_{k=1}^n \psi^{\prime}(k)\chi_E(k) }{\sum_{k=1}^n \psi^{\prime}(k)},
	\end{eqnarray}
where $\mathbb{N}_n=\{1, \ldots, n\}$. Assuming further that $\psi$ satisfies the asymptotic identity
	\begin{align}\label{asym}
	\sum_{k=1}^n\psi'(k)\sim \psi(n),\quad n\to\infty,
	\end{align} 
we then have
	\begin{eqnarray}\label{newdef2}
	\psi-\operatorname{dens}(E)=\lim _{n \rightarrow \infty} \frac{\sum_{k=1}^n \psi^{\prime}(k)\chi_E(k) }{\psi(n)},
	\end{eqnarray}
which  bears a stronger resemblance to \eqref{Barry} than to \eqref{w-density}.

\medskip
Motivated by \eqref{newdef} and \eqref{newdef2}, we propose the following definition for the weighted density of integer sets that is based on strictly increasing differentiable functions, typically satisfying \eqref{asym}, which takes the role of the second condition in \eqref{ulam}.

\begin{definition}\label{Def1}
Let $\psi: (0,\infty) \to (0,\infty)$ be a strictly increasing, differentiable, and unbounded function. The \textit{lower and upper $\psi$-densities} of a set $A\subset\N$ are defined, respectively, by
\begin{eqnarray*}
	\underline{\operatorname{d}}_\psi(A)
	=\liminf_{n\to\infty}\frac{\sum_{k=1}^n\psi'(k)\chi_A(k)}{\sum_{k=1}^n\psi'(k)}	
	\quad\textnormal{and}\quad
\overline{\operatorname{d}}_\psi(A)
	=\limsup_{n\to\infty}\frac{\sum_{k=1}^n\psi'(k)\chi_A(k)}{\sum_{k=1}^n\psi'(k)}.
	\end{eqnarray*}
If $\underline{\operatorname{d}}_\psi(A)= \overline{\operatorname{d}}_\psi(A)$, then their common value $\operatorname{d}_\psi(A)$ is called the \textit{$\psi$-density} of $A$. Additionally, if $\psi$ satisfies \eqref{asym}, then $\operatorname{d}_\psi(A)$ is called the \textit{asymptotic $\psi$-density} of $A$.
\end{definition}

Based on the fact that $\psi'(x)>0$ for all $x>0$, it is evident that for any set $A\subset\mathbb{N}$, the following inequalities hold
	$$
	0\leq \underline{\operatorname{d}}_\psi(A)\leq \overline{\operatorname{d}}_\psi(A)\leq 1.
	$$
In addition, we have 
	$$
\overline{\operatorname{d}}_\psi(A^c)+\underline{\operatorname{d}}_\psi(A)=1,
	$$ 
where $A^c$ represents the complement of $A$ in $\mathbb{N}$. Note that both the $\psi$-density and the asymptotic $\psi$-density satisfy the axiomatic definition of a density, as proposed in \cite{Grekos}. We point out here that $\psi$-density and the asymptotic $\psi$-density can be defined without requiring $\psi$ to be unbounded. For instance, consider the function $$\psi(x)=\frac{1}{e-1}-e^{-x},$$
which is a positive, strictly increasing, differentiable, and bounded function on $[1,\infty)$. It also satisfies \eqref{asym}. Such functions may lead to the corresponding $\psi$-density and asymptotic $\psi$-density not satisfying the finiteness axiom, which requires that for any finite subset $A$ of $\mathbb{N}$, we have $\operatorname{d}_{\psi}(A) = 0$.

Observe that the asymptotic $\psi$-density aligns with the asymptotic (linear) and logarithmic densities when $\psi(x)=x$ and $\psi(x)=\log x$, respectively. 
In these two cases, the asymptotic $\psi$-density gives a better justification for the names of the densities than the weighted $w$-density. 
Also note that the logarithmic density is in fact an asymptotic $\log$-density, because it is well known from elementary real analysis that the function $\psi(x)=\log x$ satisfies \eqref{asym}.

The rest of this paper is organized as follows. In Section~\ref{S2} we introduce two classes of functions, $\D_1$ and $\D_2$, based on concavity and convexity, such that the functions $\psi$ in these classes satisfy \eqref{asym}. We then proceed to prove a generalization of \eqref{asymptotic-logarithmic} for asymptotic $\psi$-densities, and also discuss the equivalence in existence of two given asymptotic $\psi$-densities. In Section~\ref{increasing-sequences-sec} we discuss the asymptotic $\psi$-density of strictly increasing sequences. In Section~\ref{S3} we consider the regularity of asymptotic $\psi$-densities for functions $\psi$ in the classes $\D_1$ and~$\D_2$. 
A function $\psi$ and the corresponding asymptotic $\psi$-density $\operatorname{d}_\psi$ are called regular, if $\operatorname{d}_\psi(\mathcal{A})=1/a$ for any integer set $\mathcal{A}$ that consists of arithmetic progressions of the form $an+b$, where $a$ and $b$ are fixed positive integers.
This paper is concluded by two appendices that contain clarifying discussions as well as alternative proofs of known results involving analytic and Abel densities.


\section{Asymptotic $\psi$-densities}\label{S2}

We introduce a function class $\mathcal{D}_1$ of concave functions with certain properties, and a function class $\mathcal{D}_2$ of convex functions with corresponding properties. The functions $\psi$ in these two classes satisfy \eqref{asym}, which defines the concept of asymptotic $\psi$-density. We then proceed to give a comparison between asymptotic densities for two distinct weight functions $\psi$ and $\varphi$, and study the equivalence of these two densities.

\subsection{Function classes $\mathcal{D}_1$ and $\mathcal{D}_2$} 

We begin by defining a class $\mathcal{D}$ of strictly increasing, differentiable, and unbounded functions $\psi: (0,\infty) \to (0,\infty)$. For convenience, we set
	\begin{equation}\label{notation-sum}
	A_{\psi }(n):=\sum_{k=1}^n\psi'(k)\chi_A(k)
	\end{equation}
for any subset $A \subset \N$. Then,
if $\psi \in \mathcal{D}$ satisfies \eqref{asym}, the lower and upper asymptotic $\psi$-densities can be expressed more briefly as
	$$
	\underline{\operatorname{d}}_{\psi}(A)
	=\liminf_{n\to\infty}\frac{A_{\psi}(n)}{\psi(n)} 
	\quad \text{and} \quad
	\overline{\operatorname{d}}_\psi(A)
	=\limsup_{n\to\infty}\frac{A_{\psi}(n)}{\psi(n)}.
	$$
At times, we consider the above $A_\psi$ as a step function, defined for the real variable $x$ by
\[
A_\psi(x) =
\begin{cases}
	\sum_{1\leq k\leq x}\psi'(k)\chi_A(k), & \text{if } x\geq 1, \\
	0, & \text{if } 0\leq x<1.
\end{cases}
\]

We proceed to discuss which strictly increasing, differentiable, and unbounded functions $\psi$ satisfy \eqref{asym}. To this end, we introduce two classes of functions based on concavity and convexity. 

\begin{definition}
Let $\D_1$ denote the class of all concave functions in $\D$, and let $\D_2$ denote the class of all convex functions $\psi$ in $\D$ that satisfy
	\begin{equation}\label{n+1-asymp-n}
	\psi(n+1)\sim \psi(n),\quad n\to\infty.
	\end{equation}
\end{definition}

The classes $\D_1$ and $\D_2$ are certainly not empty: The function $\psi_1(x)=\log (1+x)$ belongs to $\D_1$, while $\psi_2(x)=x^2$ belongs to $\D_2$.

The reason for the interval $(0,\infty)$ to be the domain of definition for functions in $\D$ is nothing but technical. It would be acceptable to replace this interval by $(r_0,\infty)$, where $r_0>0$ is some constant. This would only cause some minor modifications in the reasoning below. Hence, the function $\psi(x)=\log x$, for example, would be acceptable for the class $\D_1$. The assumption $\psi(x)>0$ cannot be ignored, however. 

Lemma~\ref{asymptotic-lemma} below shows that all functions $\psi\in\D_1$ satisfy \eqref{n+1-asymp-n}. However, one should note that not all convex functions satisfy \eqref{n+1-asymp-n}, as illustrated by the counter-example $\psi(x)=e^x$.

\begin{lemma}\label{asymptotic-lemma}
If $\psi\in\D_1$, then \eqref{asym} and \eqref{n+1-asymp-n} both hold. If $\psi\in\D_2$, then \eqref{asym} holds.
\end{lemma}

\begin{proof}
	Suppose that $\psi\in\D_1$. Then $\psi'$ is a decreasing function, and, by the endpoint rules of integral calculus,
	\begin{eqnarray*}
		\psi(n+1)-\psi(1)&=&\int_1^{n+1}\psi'(x)\, dx\leq\sum_{k=1}^n\psi'(k)\\
		&\leq& \psi'(1)+\int_1^n\psi'(x)\, dx=\psi(n)+\psi'(1)-\psi(1).
	\end{eqnarray*}
	Since the upper bound is $\leq \psi(n+1)+\psi'(1)-\psi(1)$, we obtain \eqref{asym} and \eqref{n+1-asymp-n}. 
	
Suppose then that $\psi\in\D_2$. Then $\psi'$ is an increasing function, and
	\begin{eqnarray*}
		\psi(n)+\psi'(1)-\psi(1)&=&\psi'(1)+\int_1^{n}\psi'(x)\, dx\leq\sum_{k=1}^n\psi'(k)\\
		&\leq& \int_1^{n+1}\psi'(x)\, dx=\psi(n+1)-\psi(1).
	\end{eqnarray*}
The assertion \eqref{asym} follows from this via the assumption \eqref{n+1-asymp-n}.
\end{proof}

\begin{remark}
Lemma~\ref{asymptotic-lemma} ensures that $\operatorname{d}_\psi$ is an asymptotic $\psi$-density for functions $\psi\in \D_1\cup\D_2$. In fact, this property justifies the definition of the classes $\D_1$ and $\D_2$.
\end{remark}

\begin{example}\label{p-sum-ex}
Using Lemma~\ref{asymptotic-lemma}, we can obtain the known asymptotic growth of $p$-sums. Indeed, let $\psi(x)=x^{p+1}$ for $x>0$ and $p>-1$. Then $\psi\in \D_1$ for $-1<p\leq 0$ and $\psi\in\D_2$ for $p\geq 0$. Moreover, by Lemma~\ref{asymptotic-lemma},
	$$
	\sum_{k=1}^n k^p=\frac{1}{p+1}\sum_{k=1}^n\psi'(k)\sim \frac{\psi(n)}{p+1}=\frac{n^{p+1}}{p+1},\quad p>-1.
	$$
\end{example}

For $\psi\in\D_1$, we can replace $n$ by a real variable $x$ in the proof of Lemma~\ref{asymptotic-lemma}, and obtain
	\begin{equation}\label{x+1-asymp-x}
	\psi(x+1)\sim \psi(x),\quad x\to\infty.
	\end{equation}
To prove that \eqref{x+1-asymp-x} holds for functions $\psi\in \D_2$ also, we first require a lemma that will be needed several times later on.

\begin{lemma}\label{tends-to-zero}
Let $\psi:(0,\infty)\to (0,\infty)$ be an increasing, differentiable and convex function. Then 
\begin{itemize}
\item[\textnormal{(a)}] 
\eqref{n+1-asymp-n} holds if and only if $\psi'(n)/\psi(n)\to 0$ as $n\to\infty$,
\item[\textnormal{(b)}]
\eqref{x+1-asymp-x} holds if and only if $\psi'(x)/\psi(x)\to 0$ as $x\to\infty$.
\end{itemize}
\end{lemma}

\begin{proof}
It suffices to prove (b) because the proof of (a) is the same.

Assume that $\psi(x+1)\sim \psi(x)$. By utilizing the convexity of $\psi$ and the mean value theorem, we have 
	\begin{eqnarray}\label{eq2}
	\psi'(x)\leq \psi(x+1)-\psi(x)\leq \psi'(x+1),
	\end{eqnarray}
for any $x\in (0,\infty)$. Dividing both sides of \eqref{eq2} by $\psi(x)$, we obtain
	$$
	0\leq \frac{\psi'(x)}{\psi(x)} \leq \frac{\psi(x+1)}{\psi(x)}-1,
	$$
and the assertion $\psi'(x)/\psi(x)\to 0$ as $x\to\infty$ follows by the squeeze theorem. 

Conversely, assume that $\psi'(x)/\psi(x)\to 0$ as $x\to\infty$. Dividing both sides of \eqref{eq2} by $\psi(x+1)$ yields
	$$
	\frac{\psi'(x+1)}{\psi(x+1)}\geq 1-\frac{\psi(x)}{\psi(x+1)}\geq 0.
	$$
By letting again $x\to \infty$, we derive \eqref{x+1-asymp-x}.
\end{proof}

\begin{lemma}\label{x-to-infty}
If $\psi\in\D_2$, then \eqref{x+1-asymp-x} holds.
\end{lemma}

\begin{proof}
Suppose on the contrary to the assertion that \eqref{x+1-asymp-x} does not hold. Then, by Lemma~\ref{tends-to-zero}, the logarithmic derivative $\psi'(x)/\psi(x)$ does not tend to zero along the real numbers, even though $\psi'(n)/\psi(n)$ does tend to zero along the natural numbers. Therefore, there is a strictly increasing and unbounded sequence $(x_n)_n$ of positive real numbers and a constant $\delta>0$ such that
	$$
	\psi'(x_n)/\psi(x_n)>\delta,\quad n\in\N.
	$$
Clearly, the sequence $(x_n)_n$ has at most a finite number of points $x_n$ that are natural numbers. We may assume that these points, if any, have been discarded from the sequence $(x_n)_n$. Now, for every $x_n$ there is a unique $k_n\in\N$ such that $x_n\in (k_n,k_n+1)$. Then, by the monotonicity and convexity of $\psi$,
	$$
	0<\delta<\frac{\psi'(x_n)}{\psi(x_n)}\leq \frac{\psi'(k_n+1)}{\psi(k_n)}\to 0,\quad n\to\infty,
	$$
which is a contradiction. Thus \eqref{x+1-asymp-x} holds.
\end{proof}

A further discussion on functions $\psi$ satisfying  \eqref{x+1-asymp-x} will be given in Section~\ref{regularity-D2} and in Appendix~\ref{App} below. We show, among other things, that there are strictly increasing, differentiable and convex functions $\psi$ that do not satisfy \eqref{x+1-asymp-x}, and also ponder the strength of the asymptotic identity \eqref{x+1-asymp-x} for certain functions $\psi$. This discussion can be of independent~interest.

\subsection{Comparison of asymptotic $\psi$-densities}

We proceed with recalling the concepts of relative convexity and concavity. Let $\psi$ and $\varphi$ be functions in $\mathcal{D}$. We say that $\varphi$ is convex with respect to $\psi$ if the function $\varphi \circ \psi^{-1}$ is convex. Equivalently, $\varphi$ is convex with respect to $\psi$ (or $\psi$ is concave with respect to $\varphi$) if the quotient $\varphi'(x)/\psi'(x)$ is increasing for all $x \in (0,\infty)$, as shown in \cite[Theorem 1]{C}. In other words, if both $\psi$ and $\varphi$ are twice differentiable functions, then it follows that
	$$
	\left( \frac{\psi'(x)}{\varphi'(x)}\right) ' \leq 0, \quad x\in (0,\infty).
	$$

We are now ready to present a result which is similar to \eqref{inq1} and \eqref{inq2}.

\begin{theorem}\label{Th1}
Let $A\subset\N$, and suppose that $\psi, \varphi\in\mathcal{D}$ satisfy \eqref{asym}. If $\psi$ is concave with respect to $\varphi$, then
	\begin{align}\label{key}
	0 \leq \underline{\operatorname{d}}_{\varphi}(A)\leq \underline{\operatorname{d}}_{\psi}(A) \leq \overline{\operatorname{d}}_{\psi}(A)\leq \overline{\operatorname{d}}_{\varphi}(A)\leq 1.
	\end{align}
\end{theorem}

As per the terminology introduced in the literature, the two densities satisfying \eqref{key} are comparable. More specifically, the $\varphi$-density is stronger than the $\psi$-density, and the $\psi$-density extends the $\varphi$-density.

For the convenience of the reader, we present two distinct proofs for Theorem~\ref{Th1}. The first proof utilizes Rajagopal's result, while the second one relies on Riemann-Stieltjes integration, but requires the assumption that $\psi$ and $\varphi$ are both twice differentiable. The technique used in the second proof will be applied later on in this~paper.

\begin{proof}[First proof of Theorem~\ref{Th1}]
Given that $\psi$ and $\varphi$ belong to $\D$ and satisfy \eqref{asym}, we observe that
	$$
	\sum_{k=1}^n\psi'(k)\to \infty 
	\quad \text{and} \quad 
	\sum_{k=1}^n\varphi'(k)\to \infty, \quad n\to 	\infty.
	$$
On the other hand, due to the fact that $\psi$ is concave with respect to $\varphi$, we conclude that $\psi'(x)/\varphi'(x)$ is decreasing. Therefore, by using Theorem~\ref{L1}, we establish the assertion \eqref{key}.
\end{proof}

\begin{proof}[Second proof of Theorem~\ref{Th1}]
We make an additional assumption that $\psi$ and $\varphi$ are both twice differentiable.
It suffices to prove that
	$$
	\underline{\operatorname{d}}_\varphi(A) \leq \underline{\operatorname{d}}_{\psi}(A)
	\quad\textnormal{and}\quad	
	\overline{\operatorname{d}}_{\psi}(A) \leq \overline{\operatorname{d}}_\varphi(A).
	$$
The first inequality is trivial if $\underline{\operatorname{d}}_\varphi(A)=0$, so we suppose that $\underline{\operatorname{d}}_\varphi(A)>0$. The proof of the second inequality above does not depend on this restriction.
	
Applying Riemann-Stieltjes integration to the step function $A_{\varphi}(x)$ yields
	\begin{eqnarray*}
	A_{\psi}(n)&=&\sum_{k=1}^{n}\psi'(k)\chi_A(k)
	=\sum_{k=1}^{n}\frac{\psi'(k)}{\varphi'(k)}\varphi'(k)\chi_A(k)\\
	&=& \sum_{k=1}^{n}\frac{\psi'(k)}{\varphi'(k)}\big(A_\varphi(k)-A_\varphi(k-1)\big)
	=\int_1^n\frac{\psi'(t)}{\varphi'(t)}dA_{\varphi}(t),
	\end{eqnarray*}
see \cite[Theorem 7.9]{Apostol} and \cite[Theorem 7.11]{Apostol}.
Integrating by parts, we get 
	\begin{eqnarray}\label{e}
	\begin{split}
	A_{\psi}(n)
	&=\frac{\psi'(n)}{\varphi'(n)}A_{\varphi}(n)-\int_1^n\left( \frac{\psi'(t)}{\varphi'(t)}\right) 'A_{\varphi}(t)\,dt+O(1).
	\end{split}
	\end{eqnarray}
	
Let $\varepsilon\in (0,\underline{\operatorname{d}}_{\varphi}(A))$. Then there exists a $t_0=t_0(\varepsilon)>0$ such that
	$$
	\underline{\operatorname{d}}_{\varphi}(A)-\varepsilon\leq 	\frac{A_{\varphi}(t)}{\varphi(t)}\leq \overline{\operatorname{d}}_\varphi(A)+\varepsilon, \quad t\geq t_0.
	$$
Making use of this, we obtain, for all large $n\geq t_0$, 
	\begin{eqnarray}\label{e2}
	\left( \underline{\operatorname{d}}_{\varphi}(A)-\varepsilon\right) \frac{\psi'(n)}{\varphi'(n)}\varphi(n)
	\leq\frac{\psi'(n)}{\varphi'(n)} A_\varphi(n)
	\leq \left( \overline{\operatorname{d}}_\varphi(A)+\varepsilon\right) \frac{\psi'(n)}{\varphi'(n)}\varphi(n)
	\end{eqnarray}
and 
	\begin{equation}\label{e1}
	\begin{split}
	\left( \underline{\operatorname{d}}_{\varphi}(A)-\varepsilon\right)&\int_1^n\left(- \frac{\psi'(t)}{\varphi'(t)}\right) '\varphi(t)dt+O(1)
	\leq \int_1^n\left( -\frac{\psi'(t)}{\varphi'(t)}\right) 'A_{\varphi}(t)dt \\
	& \leq \left( \overline{\operatorname{d}}_\varphi(A)+\varepsilon\right)\int_1^n\left(- \frac{\psi'(t)}{\varphi'(t)}\right) '\varphi(t)dt+O(1),
	\end{split}
	\end{equation}
where $\left(- \frac{\psi'(t)}{\varphi'(t)}\right) '\varphi(t) \geq 0$ for all $t>0$. A further integration by parts gives us
	$$
	-\int_1^n\left( \frac{\psi'(t)}{\varphi'(t)}\right) '\varphi(t)dt=\psi(n)\left(1 -\frac{\psi'(n)}{\psi(n)}\frac{\varphi(n)}{\varphi'(n)}\right)  +O(1).
	$$
Substituting this into \eqref{e1} leads to
	\begin{equation*}
	\begin{split}	
	\left( \underline{\operatorname{d}}_{\varphi}(A)-\varepsilon\right) &\psi(n)\left(1 -\frac{\psi'(n)}{\psi(n)}\frac{\varphi(n)}{\varphi'(n)}\right)  +O(1) \leq \int_1^n\left( -\frac{\psi'(t)}{\varphi'(t)}\right) 'A_{\varphi}(t)dt \\
	&\leq \left( \overline{\operatorname{d}}_\varphi(A)+\varepsilon\right) \psi(n)\left(1 -\frac{\psi'(n)}{\psi(n)}\frac{\varphi(n)}{\varphi'(n)}\right)  +O(1) , \quad t\geq t_0.
	\end{split}
	\end{equation*}
Combining this with \eqref{e} and \eqref{e2}, we find after simple calculations that for all large $n\geq t_0$, 
	$$
	\left( \underline{\operatorname{d}}_{\varphi}(A)-\varepsilon\right)+o(1)\leq \frac{A_{\psi}(n)}{\psi(n)} \leq  \left( \overline{\operatorname{d}}_\varphi(A)+\varepsilon\right)+o(1).
	$$
Taking limit inferior on the first inequality and limit superior on the second inequality gives 
	$$
	\underline{\operatorname{d}}_{\varphi}(A)-\varepsilon \leq \underline{\operatorname{d}}_{ \psi}(A)
	\quad\textnormal{and}\quad	
	\overline{\operatorname{d}}_{\psi}(A) \leq\overline{\operatorname{d}}_{\varphi}(A)+\varepsilon.
	$$
The assertions follow by letting $\varepsilon\to 0^+$.
\end{proof}	

Since any concave function $\psi$ is concave with respect to the identity function, it follows that the asymptotic density can be thought of as the ``\textit{terminal density}'' for the asymptotic $\psi$-densities when $\psi \in \D_1$. Likewise, any convex function $\varphi$ is convex with respect to the identity function. Therefore, the asymptotic density is the ``\textit{initial density}'' for the asymptotic $\varphi$-densities when $\varphi \in \D_2$. These observations will be made more precise in the following immediate consequence of Theorem~\ref{Th1}.

\begin{corollary}
The following assertions hold for any $A \subset\N$. 
	\begin{itemize}
		\item[\textnormal{(a)}] If $\psi \in \D_1$, then
		$$
		0\leq \underline{\operatorname{d}}(A) \leq \underline{\operatorname{d}}_\psi(A) \leq \overline{\operatorname{d}}_\psi(A) \leq \overline{\operatorname{d}}(A) \leq 1.
		$$
		\item[\textnormal{(b)}] If $\varphi \in \D_2$, then
		$$
		0\leq \underline{\operatorname{d}}_{ \varphi}(A) \leq \underline{\operatorname{d}}(A) \leq \overline{\operatorname{d}}(A)\leq \overline{\operatorname{d}}_{\varphi}(A) \leq 1.
		$$
	\end{itemize}
\end{corollary}

\subsection{Comparison of other densities}\label{equivalence-sec}

For $p>0$, set
	$$
	\zeta(p):=\sum_{n=1}^\infty\frac{1}{n^p}.
	$$
It is well-known that $\zeta(p)<\infty$ for all $p>1$ and that $\zeta(p)=\infty$ for all $0<p\leq 1$. The \textit{analytic density} of a set $A\subset\N$ is defined in \cite[p.~418]{Tenenbaum} as
	\begin{eqnarray*}\label{analy}
	\operatorname{d}_a(A):=\lim_{p\to 1^+}\frac{1}{\zeta(p)}\sum_{n=1}^\infty \frac{\chi_A(n)}{n^p},
	\end{eqnarray*}
provided that the limit exists. Alternatively (see Appendix~\ref{B}), we may write
	$$
	\operatorname{d}_a(A)=\lim_{p\to 1^+}(p-1)\sum_{n=1}^\infty \frac{\chi_A(n)}{n^p}.
	$$
The analytic density is closely related to the logarithmic density in the following sense.

\begin{theorem}[{\cite[p.~418]{Tenenbaum}}]\label{Tenenbaum-thm}
Let $A\subset\N$. Then the analytic density of $A$ exists if and only if the logarithmic density of $A$ exists. In this case, the two densities are equal.
\end{theorem}

We present an alternative proof for Theorem~\ref{Tenenbaum-thm} in Appendix~\ref{B}. More generally, it is natural to ask the following question, which will be answered later in this section.

\begin{question}
Under what conditions on $\psi, \varphi\in \mathcal{D}$,  the $\psi$-density of $A$ exists if and only if the $\varphi$-density of $A$ exists? If the two densities exist, are they equal?
\end{question}

The above equivalence fails to hold between asymptotic and logarithmic densities. Indeed, it is known \cite[p.~417]{Tenenbaum} that the existence of the logarithmic density does not imply the existence of the asymptotic density. Conversely, the existence of the asymptotic density does imply the existence of the logarithmic density by \eqref{asymptotic-logarithmic}. 

The next result gives a comparison between different asymptotic $\psi$-densities and contains \cite[Proposition 1]{MS} as a special case.

\begin{theorem}\label{Equiv}
Let $A\subset\N$, and let $\psi$ and $\varphi$ be functions in $\mathcal{D}$ that satisfy \eqref{asym}. If
	\begin{eqnarray}\label{Hspital2}
	0<c_1=\liminf_{x \rightarrow \infty}\frac{\psi'(x)}{\varphi'(x)} \leq \limsup_{x \rightarrow \infty}\frac{\psi'(x)}{\varphi'(x)}=c_2<\infty,
	\end{eqnarray}
then
	\begin{eqnarray}\label{result}
	\frac{1}{c} \cdot \underline{\operatorname{d}}_{\varphi}(A)\leq  \underline{\operatorname{d}}_{\psi}(A)\leq c \cdot  \underline{\operatorname{d}}_{\varphi}(A) \quad \text{and} \quad \frac{1}{c} \cdot\overline{\operatorname{d}}_{\varphi}(A) \leq   \overline{\operatorname{d}}_{\psi}(A) \leq c \cdot  \overline{\operatorname{d}}_{\varphi}(A),
	\end{eqnarray}
where $c=c_2/c_1$. In particular, if $c=1$, then
$$
	\underline{\operatorname{d}}_{\psi}(A)=\underline{\operatorname{d}}_{\varphi}(A) \quad  \text{and} \quad \overline{\operatorname{d}}_{\psi}(A)=\overline{\operatorname{d}}_{\varphi}(A).
	$$
\end{theorem}

\begin{proof}
Using \eqref{Hspital2} and \cite[Theorem II]{Taylor}, we obtain
	$$
	0<c_1=\liminf_{x \rightarrow \infty}\frac{\psi'(x)}{\varphi'(x)}\leq\liminf_{x \rightarrow \infty}\frac{\psi(x)}{\varphi (x)}\leq \limsup_{x \rightarrow \infty}\frac{\psi(x)}{\varphi (x)}\leq\limsup_{x \rightarrow \infty}\frac{\psi'(x)}{\varphi'(x)}=c_2<\infty.
	$$
Let $\varepsilon \in (0, c_1)$. Then there exists an $N = N(\varepsilon) > 0$ such that
	\begin{eqnarray}\label{inq7}
	c_1-\varepsilon \leq \frac{\psi(x)}{\varphi(x)} \leq c_2+\varepsilon \quad \text{and}\quad c_1-\varepsilon \leq \frac{\psi'(x)}{\varphi'(x)} \leq c_2+\varepsilon,\quad x\geq N.
	\end{eqnarray}
Hence, 
	$$
	\frac{1}{ c_2+\varepsilon} \leq \frac{\varphi(x)}{\psi(x)} \leq \frac{1}{c_1-\varepsilon} \quad \text{and}\quad \frac{1}{ c_2+\varepsilon} \leq \frac{\varphi'(x)}{\psi'(x)} \leq \frac{1}{c_1-\varepsilon},\quad x\geq N.
	$$
Now, let $A \subset \mathbb{N}$ be a set. Then
	\begin{eqnarray}\label{inq5}
	A_{\varphi}(n)&=&\sum_{k=N}^{n}\varphi'(k)\chi_A(k)+O(1) \leq \frac{1}{c_1-\varepsilon}\sum_{k=N}^{n} \psi'(k)\chi_A(k)+O(1)\nonumber\\
	&=& \frac{1}{c_1-\varepsilon} A_{\psi}(n)+O(1),\quad n\geq N,\nonumber
	\end{eqnarray}
and similarly,
	\begin{eqnarray}\label{inq6}
	A_{\varphi}(n) \geq \frac{1}{ c_2+\varepsilon}A_{\psi}(n)+O(1),\quad n\geq N.\nonumber
	\end{eqnarray}
Combining these estimates with \eqref{inq7}, we obtain
	$$
	\left( \frac{c_1-\varepsilon}{ c_2+\varepsilon}\right) \frac{A_{\psi}(n)}{\psi(n)}+o(1)\leq\frac{ A_{\varphi}(n)}{\varphi(n)} \leq \left( \frac{c_2+\varepsilon}{c_1-\varepsilon}\right)  \frac{A_{\psi}(n)}{\psi(n)}+o(1), \quad n\geq N.
	$$	
By using the properties of limit inferior and limit superior, we get
	$$
	\left( \frac{c_1-\varepsilon}{ c_2+\varepsilon}\right)\underline{\operatorname{d}}_{\psi}(A)\leq \underline{\operatorname{d}}_{\varphi}(A)\leq \overline{\operatorname{d}}_{\varphi}(A) \leq \left( \frac{c_2+\varepsilon}{c_1-\varepsilon}\right) \overline{\operatorname{d}}_{\psi}(A),
	$$
and by letting $\varepsilon \to 0^+$, we obtain
	\begin{eqnarray}\label{result1}
	\left( \frac{c_1}{ c_2}\right) \underline{\operatorname{d}}_{\psi}(A)\leq \underline{\operatorname{d}}_{\varphi}(A)\leq \overline{\operatorname{d}}_{\varphi}(A) \leq \left( \frac{c_2}{ c_1}\right)  \overline{\operatorname{d}}_{\psi}(A).
	\end{eqnarray}
Now, by making use of \eqref{inq7} again, one may see that
	\begin{eqnarray}
	0<\frac{1}{c_2}\leq \liminf_{x \rightarrow \infty}\frac{\varphi'(x)}{\psi'(x)}\leq \limsup_{x \rightarrow \infty}\frac{\varphi'(x)}{\psi'(x)} \leq \frac{1}{c_1}<\infty.\nonumber
	\end{eqnarray}
Hence, by using the same argument as above, we obtain
	\begin{eqnarray}\label{result2}
	\left( \frac{c_1}{ c_2}\right) \underline{\operatorname{d}}_{\varphi}(A)\leq \underline{\operatorname{d}}_{\psi}(A)\leq \overline{\operatorname{d}}_{\psi}(A) \leq \left( \frac{c_2}{ c_1}\right)  \overline{\operatorname{d}}_{\varphi}(A).
	\end{eqnarray}
Combining \eqref{result1} and \eqref{result2}, we conclude \eqref{result}.
\end{proof}

The following corollary follows directly from Theorem~\ref{Equiv}, and it gives an answer to Question 1.

\begin{corollary}\label{equiv 1}
Assume that $A\subset\N$, and let $\psi, \varphi\in\D$ be such that
	\begin{equation}\label{Hospital-condition}
	\lim_{x\to\infty}\frac{\psi'(x)}{\varphi'(x)}=C\in (0,\infty).
	\end{equation}
Then $\operatorname{d}_\psi(A)$ exists if and only if $\operatorname{d}_\varphi(A)$ exists. In this case, the two densities are equal.
\end{corollary}

As mentioned above, the existence of the logarithmic density does not imply the existence of the asymptotic density. Clearly, the assumption \eqref{Hospital-condition} fails to hold if $\psi(x)=x$ and $\varphi(x)=\log x$ (or the other way around).


\section{Asymptotic $\psi$-density of strictly increasing sequences}\label{increasing-sequences-sec}

Let $(v_n)_n$ be a strictly increasing sequence of positive integers, and let $\mathcal{S}=\{v_n : n\in\N\}$ be the associated set. It is known \cite{NZM,P} that the lower and upper asymptotic density of $\mathcal{S}$ can be expressed as
	\begin{eqnarray}\label{seq}
	\underline{\operatorname{d}}(\mathcal{S})=\liminf_{n\to\infty}\frac{n}{v_n} \quad 	\text{and} \quad
	\overline{\operatorname{d}}	(\mathcal{S})=\limsup_{n\to\infty}\frac{n}{v_n}.
	\end{eqnarray}

Motivated by this, we ask the following question.

\begin{question}
What is the asymptotic $\psi$-densities analogue of \eqref{seq}?
\end{question}

The following result addresses this question when $\psi\in\mathcal{D}_1 \cup \D_2$, and is inspired by the proofs of \cite[Theorem 11.1]{NZM} and \cite[Theorem 2.3]{P}.

\begin{theorem}\label{lem1}
Let $(v_n)_n$ be a strictly increasing infinite sequence of positive integers. Suppose that one of the following conditions holds:
	\begin{itemize}
		\item[\textnormal{(a)}] $\psi\in \mathcal{D}_1$,
		\item[\textnormal{(b)}] $\psi\in \mathcal{D}_2$ and $\psi'(v_n)/\psi(n)\to0$, as $n\to \infty$.
	\end{itemize}
Then	
	\begin{align}\label{lim1}
	\underline{\operatorname{d}}_{\psi}(\mathcal{S})=\liminf_{n\to\infty}\frac{\mathcal{S}_{\psi}(v_n)}{\psi(v_n)} \quad \text{and} \quad \overline{\operatorname{d}}_{\psi}(\mathcal{S})=\limsup_{n\to\infty}\frac{\mathcal{S}_{\psi}(v_n)}{\psi(v_n)}.
	\end{align}
\end{theorem}

Before we proceed with the proof, we need to establish an auxiliary result.

\begin{lemma}\label{lem2}
Let $(v_n)_n$ be a strictly increasing infinite sequence of positive integers. If $v_n\neq n$ for some $n$, then there exists an $N\in \mathbb{N}$ such that $v_n>n$ for all $n\geq N$.
\end{lemma}

\begin{proof}
Suppose on the contrary to the assertion that there exists a strictly increasing infinite sequence of positive integers $(n_k)_k$ such that $v_{n_k}\leq n_k$.
Note that $v_n \geq n$ for all $n$, as $(v_n)_n$ is a strictly increasing infinite sequence of positive integers. Therefore,
$
v_{n_k}=n_k,
$
for all integers $k$.
Making use of this, we get
	$$
	n_k-1\leq v_{n_k-1}<n_k,
	$$
which implies that $v_{n_k-1}=n_k-1$. Inductively, we conclude that for any $n_k$, we have $v_j=j$ for all $j\leq n_k$. Since $(n_k)_k$ is strictly increasing, we find that $v_n=n$ for all $n$, which is a contradiction.
\end{proof}

\begin{proof}[Proof of Theorem~\ref{lem1}]
	Without loss of generality, we assume that $v_n\neq n$. Clearly, since $(v_n)_n$ is a strictly increasing sequence of positive integers,
	$$
	\left( \frac{\mathcal{S}_{\psi}(v_n)}{\psi(v_n)}\right)_n \, \text{is a sub-sequence of }\;
	\left( 	\frac{\mathcal{S}_{\psi}(n)}{\psi(n)}\right) _n.
	$$
	Hence 
	$$
	\liminf_{n\to\infty}\frac{\mathcal{S}_{\psi}(v_n)}{\psi(v_n)} \geq 	\liminf_{n\to\infty} 	\frac{\mathcal{S}_{\psi}(n)}{\psi(n)} \quad \text{and}\quad 	\limsup_{n\to\infty}\frac{\mathcal{S}_{\psi}(v_n)}{\psi(v_n)} \leq 	\limsup_{n\to\infty} 	\frac{\mathcal{S}_{\psi}(n)}{\psi(n)}.
	$$
	
 Regarding the first identity in \eqref{lim1}, it remains to prove that	
	$$
	\liminf_{n\to\infty}\frac{\mathcal{S}_{\psi}(v_n)}{\psi(v_n)} \leq 	\liminf_{n\to\infty} 	\frac{\mathcal{S}_{\psi}(n)}{\psi(n)}.
	$$
	To this end, let $(m(n))_n$ be a sequence of positive integers such that $v_{m(n)}$ is the smallest integer in $\mathcal{S}$ that exceeds $n$, that is, $v_{m(n)-1}\leq n < v_{m(n)}$. Note that $(m(n))_n$ is an increasing sequence (not necessarily strictly increasing) and tends to $\infty$ as $n\to\infty$. Then 
	$$
	\mathcal{S}_{\psi}(n)=\sum_{k=1}^n\psi'(k)\chi_\mathcal{S}(k)=\sum_{k=1}^{m(n)-1}\psi'(v_k)
	$$
	and
	\begin{align*}
	\frac{\mathcal{S}_{\psi}(v_{m(n)})}{\psi(v_{m(n)})}-\frac{\mathcal{S}_{\psi}(n)}{\psi(n)}=&\frac{\mathcal{S}_{\psi}(v_{m(n)})}{\psi(v_{m(n)})}-\frac{\sum_{k=1}^{m(n)-1}\psi'(v_k)}{\psi(n)} 
	\leq\frac{\mathcal{S}_{\psi}(v_{m(n)})}{\psi(n)}-\frac{\sum_{k=1}^{m(n)-1}\psi'(v_k)}{\psi(n)}\\
	=&\frac{\sum_{k=1}^{m(n)}\psi'(v_k)-\sum_{k=1}^{m(n)-1}\psi'(v_k)}{\psi(n)}
	= \frac{\psi'(v_{m(n)})}{\psi(n)}.
	\end{align*}
	It follows that
	\begin{eqnarray}\label{eq1}
	\frac{\mathcal{S}_{\psi}(v_{m(n)})}{\psi(v_{m(n)})} \leq \frac{\mathcal{S}_{\psi}(n)}{\psi(n)}+ \frac{\psi'(v_{m(n)})}{\psi(n)}.
	\end{eqnarray}
	Let $(k(n))_n$ be the strictly increasing sub-sequence of $(m(n))_n$ consisting of all distinct elements in the sequence $(m(n))_n$. We distinguish two cases.
	\begin{itemize}
		\item [(i)] If $\psi\in \D_1$, then the derivative $\psi'$ is non-increasing, and hence $\psi'$ is bounded from above. This, together with \eqref{eq1}, implies 
		$$
		\liminf_{n\to\infty}\frac{\mathcal{S}_{\psi}(v_n)}{\psi(v_n)} \leq \liminf_{n\to\infty}\frac{\mathcal{S}_{\psi}(v_{k(n)})}{\psi(v_{k(n)})}=\liminf_{n\to\infty}\frac{\mathcal{S}_{\psi}(v_{m(n)})}{\psi(v_{m(n)})} \leq\liminf_{n\to\infty} 	\frac{\mathcal{S}_{\psi}(n)}{\psi(n)},
		$$
		and the inequality
\begin{eqnarray}\label{lim4}
		\underline{\operatorname{d}}_{\psi} (\mathcal{S}) \geq	\liminf_{n\to\infty}\frac{\mathcal{S}_{\psi}(v_n)}{\psi(v_n)} 
\end{eqnarray}
		follows immediately. 
		\item [(ii)] If $\psi\in \D_2$, then, by using Lemma~\ref{lem2} and considering that $v_{m(n)}$ is the smallest integer that exceeds $n$, we deduce the existence of $N\in \mathbb{N}$ such that 
		$$\psi'(v_{m(n)})\leq \psi'(v_n), \quad \text{for all}\;n\geq N.$$ 
This, along with the assumption $\psi'(v_n)/\psi(n)\to0$ as $n\to \infty$, yields 
	$$
	\lim_{n\to\infty}\frac{\psi'(v_{m(n)})}{\psi(n)}=0.
	$$
Consequently, by applying a similar proof to that in Case (i), we obtain \eqref{lim4}.
	\end{itemize}

	Regarding the second identity in \eqref{lim1}, it remains to prove that
	$$
	\limsup_{n\to\infty}\frac{\mathcal{S}_{\psi}(v_n)}{\psi(v_n)} \geq 	\limsup_{n\to\infty} 	\frac{\mathcal{S}_{\psi}(n)}{\psi(n)}=\overline{\operatorname{d}}_{\psi}(\mathcal{S}).
	$$
Clearly, we may suppose that $\overline{\operatorname{d}}_{\psi}(\mathcal{S})>0$.
	By definition, for any given $\varepsilon\in (0,\overline{\operatorname{d}}_{\psi}(\mathcal{S}))$, there exists a strictly increasing sequence $(n_k)_k$ of positive integers such that 
	\begin{align}\label{prop}
	\frac{\mathcal{S}_{\psi}(n_k)}{\psi(n_k)}>\overline{\operatorname{d}}_{\psi}(\mathcal{S})-\varepsilon.
	\end{align}
	Let $n_{k_1}\geq v_1$ be an integer satisfying \eqref{prop}. Then there exists a positive integer $m_1$ such that $v_{m_1} \leq n_{k_1} <v_{m_{1}+1}$. Hence
	$$
	\mathcal{S}_{\psi}(n_{k_1})=\sum_{j=1}^{n_{k_1}}\psi'(j)\chi_{\mathcal{S}}(j)=\sum_{j=1}^{m_1}\psi'(v_j)=\mathcal{S}_{\psi}(v_{m_1}),
	$$
	and so,
	$$
	\frac{\mathcal{S}_{\psi}(v_{m_1})}{\psi(v_{m_1})}=\frac{	\mathcal{S}_{\psi}(n_{k_1})}{\psi(v_{m_1})} \geq \frac{	\mathcal{S}_{\psi}(n_{k_1})}{\psi(n_{k_1})}>\overline{\operatorname{d}}_{\psi}(\mathcal{S})-\varepsilon.
	$$
	Thus, there exists at least one integer $m_1$ such that
	\begin{eqnarray}\label{prop2}
	\frac{\mathcal{S}_{\psi}(v_{m_k})}{\psi(v_{m_k})}>\overline{\operatorname{d}}_{\psi}(\mathcal{S})-\varepsilon.
	\end{eqnarray}
	For any such $m_1$, there exists an integer $n_{k_2} \geq v_{m_{1}+1}$ satisfying \eqref{prop}. Similarly, there exists an integer $m_2$ such that $v_{m_2}\leq n_{k_2}<v_{m_2+1}$. Hence, $\mathcal{S}_{\psi}(n_{k_2})=\mathcal{S}_{\psi}(v_{m_2})$ and
	$$
	\frac{\mathcal{S}_{\psi}(v_{m_2})}{\psi(v_{m_2})}=\frac{	\mathcal{S}_{\psi}(n_{k_2})}{\psi(v_{m_2})} \geq \frac{	\mathcal{S}_{\psi}(n_{k_2})}{\psi(n_{k_2})}>\overline{\operatorname{d}}_{\psi}(\mathcal{S})-\varepsilon.
	$$
	Therefore, $m_2$ satisfies \eqref{prop2}. Note that since $n_{k_2} \geq v_{m_1+1}$, we have
	$$
	\mathcal{S}_{\psi}(v_{m_2})=\mathcal{S}_{\psi}(n_{k_2}) \geq \mathcal{S}_{\psi}(v_{m_1+1})>\mathcal{S}_{\psi}(v_{m_1}).
	$$	
	Thus, there exist an infinite strictly increasing sequence $(m_k)_k$ of integers that satisfies \eqref{prop2}, and consequently
	$$
	\limsup_{n\to\infty}\frac{\mathcal{S}_{\psi}(v_n)}{\psi(v_n)} \geq \overline{\operatorname{d}}_{\psi}(\mathcal{S}).
	$$
This completes the proof of Theorem~\ref{lem1}.
\end{proof}

\begin{remark}
From the proof of Theorem~\ref{lem1}, we observe that the limit superior in \eqref{lim1} remains valid for $\psi \in \mathcal{D}_2$ without the need for any additional conditions.
\end{remark}

\begin{example}
We illustrate the usefulness of Theorem~\ref{lem1} in counting asymptotic densities. Let $\mathcal{S}$ be the set of positive even integers. For $\psi(x)=\sqrt{x}\in \D_1$, we may use Lemma~\ref{asymptotic-lemma} (or Example~\ref{p-sum-ex}) to obtain
	\begin{eqnarray*}
	\mathcal{S}_{\psi}(2n)&=&\sum_{k=1}^{2n}\frac{\chi_{\mathcal{S}}(k)}{2\sqrt{k}}=\sum_{k=1}^{n}	\frac{1}{2\sqrt{2k}}=\frac{1}{\sqrt{2}}\sum_{k=1}^{n}	\psi'(k)
	\sim \frac{\psi(n)}{\sqrt{2}}= \frac{\sqrt{n}}{\sqrt{2}},\quad n\to \infty.
	\end{eqnarray*}
Therefore, by Theorem~\ref{lem1},
	$$
	\operatorname{d}_{\psi}(\mathcal{S})=\lim_{n\to\infty}\frac{\mathcal{S}_{\psi}(2n)}{\psi(2n)}=\frac{1}{2}.
	$$
Consider now $\psi(x)=x\log x \in \D_2$. Noting that 
	$$
	\frac{\psi'(2n)}{\psi(n)}=\frac{\log (2n)+1}{n\log n}\to 0,\quad n\to \infty,
	$$
we have
	$$
	\mathcal{S}_{\psi}(2n)=\sum_{k=1}^{2n}(1+\log k)\chi_{\mathcal{S}}(k)=\sum_{k=1}^{n}(1+\log (2k))=n(1+\log 2)+\log n!.
	$$
Therefore, by Stirling's formula and Theorem~\ref{lem1},
	$$
	\operatorname{d}_{\psi}(\mathcal{S})=\lim_{n\to\infty}\frac{\mathcal{S}_{\psi}(2n)}{\psi(2n)}=\lim_{n\to\infty}\frac{\log n!}{2n\log (2n)}=\frac{1}{2}.
	$$
\end{example}

\begin{remark}
    It seems that the condition $\psi'(v_n)/\psi(n)\to0$, as $n\to \infty$, is sufficient but not necessary for the conclusion of Theorem~\ref{lem1}. For example, consider the function $\psi(x)=e^{\sqrt{x}}\in \mathcal{D}_2$ satisfying
	$$
	\frac{\psi'(2n)}{\psi(n)}\to \infty,\quad n\to \infty.
	$$
If $\mathcal{S}$ denotes the set of positive even integers, it follows that
 	$$
	\mathcal{S}_{\psi}(2n+1)=\mathcal{S}_{\psi}(2n)=\sum_{k=1}^{2n}\frac{e^{\sqrt{k}}}{2\sqrt{k}}\chi_{\mathcal{S}}(k)=\sum_{k=1}^{n}\frac{e^{\sqrt{2k}}}{2\sqrt{2k}}\sim \frac{1}{2}e^{\sqrt{2n}},\quad \text{as } n\to \infty,
	$$
which yields
	$$
	\lim_{n\to\infty}\frac{\mathcal{S}_{\psi}(2n+1)}{\psi(2n+1)}=\frac{1}{2}\quad \text{and}\quad \lim_{n\to\infty}\frac{\mathcal{S}_{\psi}(2n)}{\psi(2n)}=\frac{1}{2}.
	$$
Thus $\operatorname{d}_{\psi}(\mathcal{S})=1/2$.
\end{remark}


\section{Regularity of asymptotic $\psi$-densities}\label{S3}

In this section, we will use the notation $(u_n)_n$ to represent an arithmetic sequence of positive integers. For the sake of simplicity, these sequences are defined as $u_n = an + b$, where $a\in \mathbb{N}$ and $b$ is a non-negative integer. We denote $\mathcal{A}\subset \mathbb{N}$ as any subset with elements forming an arithmetic progression of positive integers, formally expressed as  $\mathcal{A} = \{u_n : n\in\N\}$.

  A function $\psi\in \D$ and the corresponding asymptotic $\psi$-density $\operatorname{d}_\psi$ are said to be \textit{regular} if $\operatorname{d}_{\psi}(\mathcal{A})=1/a$ for any fixed integers $a$ and $b$. This definition originates from a well-known fact \cite{P} that both the asymptotic density and the logarithmic density are regular.  We arrive at the following question.

\begin{question}\label{Con1}
What conditions on $\psi \in \mathcal{D}$ determine that $\psi$ is regular?
\end{question}

We will consider regularity in the classes $\D_1$ and $\D_2$.

\subsection{Regularity in $\D_1$ and $\D_2$} 

As the main result of this subsection, we prove that all functions in $\D_1$ are regular, while functions in $\D_2$ are regular under certain additional conditions.

\begin{theorem}\label{Reg}
Let $(u_n)_n$ be an infinite arithmetic sequence of positive integers. Suppose that one of the following conditions holds:
\begin{itemize}
	\item[\textnormal{(a)}] $\psi\in \mathcal{D}_1$,
	\item[\textnormal{(b)}] $\psi\in \mathcal{D}_2$ and $\psi'(u_n)/\psi(n)\to0$, as $n\to \infty$.
\end{itemize}
Then
	$
	\operatorname{d}_{\psi}(\mathcal{A})=1/a.
	$
\end{theorem}

\begin{proof}
From Theorem~\ref{lem1}, it suffices to show that 
	\begin{eqnarray}\label{lim3}
		\lim_{n\to\infty}\frac{\mathcal{A}_{\psi}(u_n)}{\psi(u_n)}=\frac{1}{a},
	\end{eqnarray}
	where $u_n=an+b$ and 
	$$
	\mathcal{A}_{\psi}(u_n)=\sum_{k=1}^{u_n}\psi'(k)\chi_\mathcal{A}(k)=\sum_{k=1}^{n}\psi'(u_k).
	$$
By the mean value theorem, 
	\begin{align}
	\psi(u_n)=&\big(\psi(u_n)-\psi(u_{n-1})\big)+\big(\psi(u_{n-1})-\psi(u_{n-2})\big)+\cdots+\big(\psi(u_2)-\psi(u_1)\big)+\psi(u_1)\nonumber\\
	=&a\big( \psi'(c_{n,n-1})+\psi'(c_{n-1,n-2})+\cdots+\psi'(c_{2,1})\big) +\psi(u_1),\nonumber
	\end{align}
where $c_{k,k-1}\in (u_{k-1},u_k)$ for $k=2,\ldots,n$.

\begin{itemize}
	\item[(a)] If $\psi \in \D_1$, then the derivative $\psi'$ is decreasing and
	$$
	a\sum_{k=2}^{n}\psi'(u_k)+\psi(u_1)\leq \psi(u_n)\leq a\sum_{k=1}^{n-1}\psi'(u_k)+\psi(u_1),
	$$
that is,
	\begin{equation}\label{left-center-right}
	a	\frac{\mathcal{A}_{\psi}(u_n)}{\psi(u_n)}+\frac{\psi(u_1)-a\psi'(u_1)}{\psi(u_n)}\leq 1\leq a	\frac{\mathcal{A}_{\psi}(u_n)}{\psi(u_n)}-\frac{a\psi'(u_n)-\psi(u_1)}{\psi(u_n)}.
	\end{equation}
By the mean value theorem,
	$$
	0\leq \psi'(x)\leq \psi(x)-\psi(x-1),\quad x>1,
	$$
and hence
	$$
	0\leq \frac{\psi'(x)}{\psi(x)}\cdot \frac{\psi(x)}{\psi(x-1)}\leq \frac{\psi(x)}{\psi(x-1)}-1,\quad x>1.
	$$
Using \eqref{x+1-asymp-x} and the squeeze theorem, we find that $\frac{\psi'(x)}{\psi(x)}\to 0$ as $x\to\infty$.
Taking limit inferior on the second inequality in \eqref{left-center-right} and limit superior on the first inequality in \eqref{left-center-right} gives us
	$$
	\frac{1}{a}        \leq \underline{\operatorname{d}}_{\psi}(\mathcal{A})\leq\overline{\operatorname{d}}_{\psi}(\mathcal{A})\leq \frac{1}{a},
	$$
from which \eqref{lim3} follows immediately.

\item[(b)] If $\psi\in \D_2$, then $\psi'$ is increasing and
	$$
a	\mathcal{A}_{\psi}(u_n)-a\psi'(u_n)+\psi(u_1)\leq \psi(u_n)\leq 	a	\mathcal{A}_{\psi}(u_n)+\psi(u_1)-a\psi'(u_1).
$$
By making use of Theorem~\ref{lem1} and Lemma~\ref{tends-to-zero}, and deducing similarly as in Case (a), we obtain \eqref{lim3}.
\end{itemize}
This completes the proof.
\end{proof}

\subsection{More sufficient conditions for regularity in $\D_2$}\label{regularity-D2} 

As one may notice, for the condition 
\begin{eqnarray}\label{eq3}
	\frac{\psi'(u_{n})}{\psi(n)} = \frac{\psi'(u_{n})}{\psi(u_{n})} \cdot \frac{\psi(u_{n})}{\psi(n)} \to 0, \quad n \to \infty,
\end{eqnarray}
 in Theorem~\ref{Reg} to hold for $\psi \in \mathcal{D}_2$,  it suffices that $\frac{\psi(u_n)}{\psi(n)}$ is bounded, because we have $\psi'(u_n)/\psi(u_n) \to 0$ by Lemma~\ref{tends-to-zero}.  The following example shows that the condition \eqref{eq3} may not always hold.

\begin{example}\label{test-funs-ex}
	The test function $\psi(x)=x^p$, $p\geq 1$, does satisfy \eqref{eq3}, but the function $\psi(x)=\exp\big(x^\sigma\big)$, $\sigma\in (0,1)$, does not. Indeed,
	$$
	u_n^\sigma-n^\sigma=\sigma\int_n^{u_n}t^{\sigma-1}\, dt\geq \sigma u_n^{\sigma-1}(u_n-n),
	$$
and hence
	$$
	\frac{\psi'(u_n)}{\psi(n)}=\sigma u_n^{\sigma-1}\exp\big(u_n^\sigma-n^\sigma\big)\to\infty,\quad n\to\infty,
	$$
whenever $a\geq 2$ in $u_n=an+b$. 
\end{example}

From Example~\ref{test-funs-ex}, one may wonder if we have to restrict ourselves to functions $\psi$ of finite order in the sense of
	\begin{equation}\label{finite-order}
\limsup_{x\to\infty}\frac{\log\psi(x)}{\log x}<\infty.
	\end{equation}
However, \eqref{finite-order} alone is not sufficient to guarantee that $\psi(u_n)/\psi(n)$ is bounded. Indeed, in Lemma~\ref{not-ASV-lemma} below we will construct a finite order function $\psi$ for which $\psi(u_n)/\psi(n)$ is unbounded and for which $\psi(x+1)\not\sim \psi(x)$. Recall from Lemma~\ref{x-to-infty} that $\psi(x+1)\sim \psi(x)$ for all $\psi\in\D_2$.

The discussion above leads to the following question.

\begin{question}\label{Q}
What conditions on $\psi$ ensure the boundedness of  $\psi(an+b)/\psi(n)$?
\end{question}


To explore this further, let us verify that $\psi$ needs to be of finite order. Interestingly, it appears that the constant $b$ is irrelevant in this context.

\begin{lemma}\label{finite-order-lemma}
Let $\psi:(0,\infty)\to (0,\infty)$ be a strictly increasing function satisfying
	$$
	\frac{\psi(ax)}{\psi(x)}\leq C<\infty
	$$
for some integer $a\geq 2$. Then the order of $\psi$ is at most $\log C/\log a$. 
\end{lemma}

\begin{proof}
	Let $x\geq a$ be large. Then there exists an $N\in\N$ such that $a^N\leq x\leq a^{N+1}$, that is, 	$$
	N\leq \frac{\log x}{\log a}\leq N+1
	\quad\textnormal{and}\quad
	1\leq \frac{x}{a^N}\leq a.
	$$ 
	Deducing inductively,
	\begin{eqnarray*}
		\psi(ax)&\leq& C\psi(x)\leq C^2\psi\left(\frac{x}{a}\right)\leq \cdots\leq C^{N+1}\psi\left(\frac{x}{a^N}\right)\\
		&\leq& C\psi(a)\exp\big(N\log C\big)\leq C\psi(a)x^{\frac{\log C}{\log a}}.
	\end{eqnarray*}
	Thus $\psi(x)=O\left(x^{\frac{\log C}{\log a}}\right)$, from which the assertion follows.
\end{proof}

\begin{remark}
Functions $\psi\in\D_2$ are not in general of finite order, but they do have a growth restriction $\log\psi(x)=o(x)$. To see this, we first recall from Lemma~\ref{x-to-infty} that $\psi(x+1) \sim \psi(x)$ as $x\to\infty$. The assertion then follows by Lemma~\ref{ASV-implies-growth} below.
\end{remark}

Based on Example~\ref{test-funs-ex} and Lemma~\ref{finite-order-lemma}, a natural question is, \textit{how strong the asymptotic equation $\psi(x+1)\sim \psi(x)$ is for a function $\psi\in \mathcal{D}_2$?} Let us consider the test function $\psi(x)=x^p$, $p\geq 1$, for which
	$$
	\frac{\psi(x+1)}{\psi(x)}=\left(1+\frac{1}{x}\right)^p.
	$$
Consequently,
	\begin{equation}\label{condition1}
	\limsup_{x\to\infty}\frac{\log\frac{\psi(x+1)}{\psi(x)}}{\log\left(1+1/x\right)}<\infty.
	\end{equation}	 
This condition can be simplified by using L'Hospital's rule, which gives us
	$$
	\lim_{x\to\infty}\frac{\log\left(1+1/x\right)}{1/x}
	=\lim_{y\to 0^+}\frac{\log(1+y)}{y}=1.
	$$
Thus \eqref{condition1} is equivalent to
	\begin{equation}\label{condition2}
	\limsup_{x\to\infty}\left(x\log\frac{\psi(x+1)}{\psi(x)}\right)<\infty.
	\end{equation}
This condition is satisfied by $\psi(x)=x^p$, but in general it illustrates the strength of the  asymptotic equation $\psi(x+1)\sim \psi(x)$. We are mostly interested in unbounded functions $\psi$, but an example of a bounded function $\psi$ satisfying both $\psi(x+1)\sim \psi(x)$ and \eqref{condition2} is given by $\psi(x)=\arctan(x)$.

Condition \eqref{condition2} gives an answer to Question~\ref{Q}, as stated in Theorem~\ref{ratio-thm} below. It is worth noting that neither convexity nor differentiability is necessary for this result.

\begin{theorem}\label{ratio-thm}
Let $\psi:(0,\infty)\to (1,\infty)$ be a function. Then the following assertions hold.
	\begin{itemize}
		\item[\textnormal{(a)}] If \eqref{condition2} holds, then \eqref{finite-order} holds and the ratio $\frac{\psi(an+b)}{\psi(n)}$ is a bounded function of $n$ for any fixed constants $a,b\in \mathbb{N}$.
		\item[\textnormal{(b)}] If $\displaystyle\liminf_{x\to\infty}\left(\frac{x}{\log x}\log\frac{\psi(x+1)}{\psi(x)}\right)>0$, then $\displaystyle\liminf_{x\to\infty}\frac{\log\psi(x)}{\log x}=\infty$ and the ratio $\frac{\psi(ax)}{\psi(x)}$ is an unbounded function of $x$ for any fixed constant $a>1$.
	\end{itemize}
\end{theorem}

\begin{proof}
	(a) Note that both assertions are clearly true if $\psi$ is bounded. Hence we may assume that $\psi$ is unbounded. Let $c=\max\{a,b\}$. By \eqref{condition2}, there exist constants $K>0$ and $R>0$ such that
	\begin{equation}\label{consequence}
	\frac{\psi(x+1)}{\psi(x)}\leq \exp\left(\frac{K}{x}\right),\quad x>R.
	\end{equation}
Then
	\begin{eqnarray*}
		\frac{\psi(an+b)}{\psi(n)}&\leq& \frac{\psi(c(n+1))}{\psi(n)}= \frac{\psi(c(n+1))}{\psi(c(n+1)-1)}\cdot \frac{\psi(c(n+1)-1)}{\psi(c(n+1)-2)}\cdots \frac{\psi(n+1)}{\psi(n)}\\
		&\leq & \exp\left(K\sum_{k=n}^{c(n+1)-1}\frac1k\right)\leq\exp\left(\frac{Kc(n+1)}{n}\right)\leq e^{2Kc},\quad n>R,
	\end{eqnarray*}
which proves the second assertion in (a).
We now proceed to prove the first assertion in (a). To simplify the reasoning, we may suppose, without loss of generality, that the consequence \eqref{consequence} of \eqref{condition2} is valid for all $x>0$. Let $x\geq 2$ be large. Then there exists an $N\in\N$ such that $N+1\leq x\leq N+2$. Deducing inductively,
	\begin{eqnarray*}
		\psi(x+1)&\leq& \psi(x)\exp\left(\frac{K}{x}\right) \leq \psi(x-1)\exp\left(\frac{K}{x-1}+\frac{K}{x}\right)\leq\cdots\\
		&\leq & \psi(x-N)\exp\left(K\sum_{j=0}^{N}\frac{1}{x-j}\right)\leq \psi(2)\exp\left(K\sum_{j=1}^{N+1}\frac{1}{x-(j-1)}\right)\\
		&\leq& \psi(2)\exp\left(K\sum_{j=1}^{N+1}\frac{1}{j}\right)\leq \psi(2)\exp\big(O(\log x)\big),
	\end{eqnarray*}
	which implies the first assertion in (a). 
	
	(b) By the assumption, there exist constants $\delta>0$ and $R>0$ such that
	$$
	\frac{\psi(x+1)}{\psi(x)}\geq \exp\left(\frac{\delta\log x}{x}\right),\quad x>R.
	$$ 
	Recall that the function $x\mapsto\frac{\log x}{x}$ is decreasing for $x\geq e$. Hence, deducing similarly as in (a), but for a large $x$ satisfying $N+3\leq x\leq N+4$, we obtain
	$$
	\psi(x+1)\geq \psi(3)\exp\left(O\left(\log^2 x\right)\right),
	$$
proving the first assertion in (b). The second assertion in (b) follows by Lemma~\ref{finite-order-lemma} because $\psi$ is of infinite order.
\end{proof}

\begin{remark}
The assumption in Theorem~\ref{ratio-thm}(b) can be replaced with a weaker assumption
	$$
	\liminf_{x\to\infty}\left(\frac{x}{h(x)}\log\frac{\psi(x+1)}{\psi(x)}\right)>0,
	$$
where $h$ is a continuous function tending to infinity arbitrarily slowly such that $\frac{h(x)}{x}$ is decreasing. Indeed, we obtain, for some $\delta>0$ and for $N+1\leq x\leq N+2$,
	\begin{eqnarray*}
	\psi(x+1) &\geq& \psi(x)\exp\left(\frac{\delta h(x)}{x}\right)\geq \psi(x-1)\exp\left(\frac{\delta h(x)}{x}+\frac{\delta h(x-1)}{x-1}\right)\\
	&\geq& \ldots\geq \psi(x-N)\exp\left(\delta\sum_{j=0}^N\frac{h(x-j)}{x-j}\right).
	\end{eqnarray*}
Consequently, using the fact that $\frac{h(x)}{x}$ is decreasing,
	$$
	\log \psi(x+1)\geq O(1)+\delta\int_{\alpha}^{x+1}\frac{h(t)}{t}\, dt,
	$$
where $\alpha>0$ is a number independent of $x$. By L'Hospital's rule,
	$$
	\frac{\int_{\alpha}^{x+1}\frac{h(t)}{t}\, dt}{\log (x+1)}\sim h(x+1)\to\infty,\quad x\to\infty,
	$$
and we conclude that $\psi$ is of infinite order. The fact that $\frac{\psi(ax)}{\psi(x)}$ is an unbounded function of $x$ follows again by Lemma~\ref{finite-order-lemma}.
\end{remark}

\appendix


\section{On the identity $\psi(n+1)\sim \psi(n)$}\label{App}

In this appendix, we explore the possibility of convex functions $\psi$ to satisfy the crucial identity  \eqref{x+1-asymp-x}. We will show that a convex function $\psi$ of finite order, as a restriction, does not always satisfy \eqref{x+1-asymp-x}. Recall \eqref{finite-order} for the definition of order. This appendix can be of independent interest to readers interested in the  properties of convex functions.
 
Before we proceeding, let us provide an illustrative example of convex functions that do satisfy \eqref{x+1-asymp-x}.

\begin{example}
Functions $\psi(x)=x^p$, $p\geq 1$, satisfy \eqref{x+1-asymp-x}. This follows from
	$$
	px^{p-1}\leq (x+1)^p-x^p=p\int_x^{x+1}t^{p-1}\, dt\leq p(x+1)^{p-1}.
	$$
	On the other hand, if $p\in (0,1)$, then
	$$
	p(x+1)^{p-1}\leq (x+1)^p-x^p=p\int_x^{x+1}t^{p-1}\, dt\leq px^{p-1},
	$$
	and we see that $\psi(x)=\exp\big(x^p\big)$ satisfies \eqref{x+1-asymp-x}.
\end{example}	

A strictly increasing and measurable (or continuous) function $\psi:(r_0,\infty)\to (1,\infty)$ is called \emph{additively slowly varying} (or translational slowly varying)  \cite{Taskovic}, if
	\begin{equation}\label{ASV}
	\psi(x+\lambda)\sim\psi(x),\quad x\to\infty,
	\end{equation}
for all $\lambda>0$. This is a slightly stronger condition than \eqref{x+1-asymp-x}. Next we prove that \eqref{ASV} (as well as \eqref{x+1-asymp-x}) implies a growth restriction for $\psi$, keeping in mind that $\psi(x)=e^x$ does not satisfy either of \eqref{x+1-asymp-x} or \eqref{ASV}.

\begin{lemma}\label{ASV-implies-growth}
If $\psi:(r_0,\infty)\to (1,\infty)$ is additively slowly varying, then 
	\begin{equation}\label{hyper-order}
	\log\psi(x)=o(x).
	\end{equation}
\end{lemma}

\begin{proof}
Let $\varepsilon>0$ and $\lambda>0$. Then there exists a constant $R>r_0$ such that
	$$
	\psi(x+\lambda)\leq (1+\varepsilon)\psi(x),\quad x\geq R.
	$$
Let $x\geq R+\lambda$. Then there exists an $N\in\N$ such that $R+N\lambda\leq x\leq R+(N+1)\lambda$. Deducing inductively, we obtain
	$$
	\psi(x+\lambda)\leq (1+\varepsilon)^{N+1}\psi(x-N\lambda)\leq (1+\varepsilon)^{N+1}\psi(R+\lambda).
	$$
Since $N\leq \frac{x-R}{\lambda}$, we obtain
	$$
	\log\psi(x)\leq \frac{\log(1+\varepsilon)}{\lambda}x+O(1),\quad x\geq  R+2\lambda.
	$$
Consequently, $\limsup\limits_{x\to\infty}\frac{\log\psi(x)}{x}\leq \frac{\log(1+\varepsilon)}{\lambda}$, where we may let $\varepsilon\to 0^+$. 
\end{proof}

Conversely to Lemma~\ref{ASV-implies-growth} (keeping in mind that we are currently interested in convex functions), we may ask if an increasing, continuous and convex function  $\psi:(r_0,\infty)\to (1,\infty)$ satisfying a growth restriction would always satisfy \eqref{x+1-asymp-x}. Apparently, this is not the case.

\begin{lemma}\label{not-ASV-lemma}
	For every $p>1$ there exist a strictly increasing and unbounded sequence $(x_n)_n$ of positive real numbers and a strictly increasing, continuous and convex function $\psi:[x_1,\infty)\to (1,\infty)$ that is linear on each interval $[x_n,x_{n+1}]$ and satisfies
	\begin{equation}\label{limit-psi}
	\lim_{n\to\infty}\frac{\psi(x_n+1)}{\psi(x_n)}=\left\{\begin{array}{rl}1,\ & 1<p<2,\\ 2,\ & p=2,\\ \infty,\ & p>2.\end{array}\right.
	\end{equation}
	Moreover, the function $\psi$ has the growth rate
	\begin{equation}\label{growth-psi}
	\lim_{x\to\infty}\frac{\log\psi(x)}{\log x}=p,
	\end{equation}
	that is, the limit in \eqref{growth-psi} exists and is equal to $p$.
\end{lemma}

\begin{proof}
	We generalize the reasoning used in constructing \cite[Example~1]{JO}. Let $x_n=2^{p^n}$ for $n\in\N$, so that $x_{n+1}=x_n^p$. We let $\psi$ to be the piecewise linear function inscribed in $x^p$, that is, $\psi(x_n)=x_n^p$ and $\psi$ is linear on the intervals $[x_n,x_{n+1}]$ with
	\begin{equation}\label{psi-on-the-interval2}
	\begin{split}
	\psi(x)&=\psi(x_n)+\frac{\psi(x_{n+1})-\psi(x_n)}{x_{n+1}-x_n}(x-x_n)\\
	&= x_n^p+\frac{x_{n+1}^p-x_n^p}{x_{n+1}-x_n}(x-x_n),\quad x\in [x_n,x_{n+1}].
	\end{split}
	\end{equation}
	It is obvious that $\psi$ is strictly increasing, continuous and convex. Moreover, 
	$$
	\frac{\psi(x_n+1)}{\psi(x_n)}=1+\frac{x_{n+1}^p-x_n^p}{x_n^p(x_{n+1}-x_n)}\sim 1+x_n^{p(p-2)}\to\left\{\begin{array}{rl}1,\ & 1<p<2,\\ 2,\ & p=2,\\ \infty,\ & p>2.\end{array}\right.
	$$
	This proves \eqref{limit-psi}. It remains to prove \eqref{growth-psi}. To this end, let $\alpha\in [0,1]$, and denote $y_n=\alpha x_n+(1-\alpha)x_{n+1}$. Then
	$$
	\psi(y_n)=x_n^p+(1-\alpha)(x_{n+1}^p-x_n^p)
	=(1-\alpha)x_{n+1}^p+\alpha x_n^p.
	$$
	We have
	$$
	\frac{\log \psi(x_n)}{\log x_n}=p,\quad n\in\N,
	$$
	and 
	$$
	\frac{\log \psi(y_n)}{\log y_n}
	\sim\frac{\log \big((1-\alpha)x_{n+1}^p\big)}{\log \big((1-\alpha)x_{n+1}\big)}\to p,\quad n\in\N,\ \alpha\in (0,1).
	$$
	For any real $x\geq x_1$ there exists an $N\in\N$ such that $x_N\leq x\leq x_{N+1}$. In other words, $x=y_N$ for some $\alpha\in [0,1]$. The assertion  \eqref{growth-psi} follows by the continuity of $\psi$.
\end{proof} 

The growth condition \eqref{hyper-order} is of special interest as it is not satisfied by $\psi(x)=e^x$. Next, in addition to \eqref{hyper-order}, we assume that $\log\psi$ is concave, that is, $\psi$ is log-concave. Note that the asymptotic identity \eqref{n+1-asymp-n}, which is satisfied by the functions $\psi\in\D_2$ by definition, is not assumed. See also Lemma~\ref{x-to-infty} for comparison.

\begin{lemma}\label{log-concave-lemma}
Let $\psi:(r_0,\infty)\to (1,\infty)$ be a strictly increasing, differentiable, unbounded and convex function. In addition, we suppose that $\psi$ is log-concave, and that \eqref{hyper-order} holds.
Then \eqref{x+1-asymp-x} holds. 
\end{lemma}

\begin{proof}
By differentiability and log-concavity, the derivative
	$$
	(\log \psi(x))'=\frac{\psi'(x)}{\psi(x)}
	$$
is a decreasing function. Suppose that there are constants $\delta>0$ and $R>0$ such that
	$$
	\frac{\psi'(t)}{\psi(t)}>\delta,\quad t\geq R.
	$$
Integrating both sides over the interval $[R,x]$, we find that
	$$
	\log\psi(x)-\log\psi(R)\geq \delta(x-R),\quad x\geq R,
	$$
which violates the assumption \eqref{hyper-order}. Therefore, as a decreasing and positive function, the logarithmic derivative $\psi'(x)/\psi(x)$ tends to zero as $x\to\infty$. The assertion then follows by Lemma~\ref{tends-to-zero}.	
\end{proof}

\begin{remark}	
	(a) If $\psi$ is twice differentiable, then log-concavity can be expressed equivalently as
	$$
	(\log\psi(x))''=\left(\frac{\psi'(x)}{\psi(x)}\right)'=\frac{\psi''(x)}{\psi(x)}-\left(\frac{\psi'(x)}{\psi(x)}\right)^2\leq 0.
	$$
	Writing this as $\psi''/\psi'\leq \psi'/\psi$ and integrating over the interval $[a,x]$, we obtain
	$$
	\log\frac{\psi'(x)}{\psi'(a)}\leq \log \frac{\psi(x)}{\psi(a)}\quad\Longrightarrow\quad \frac{\psi'(x)}{\psi(x)}\leq \frac{\psi'(a)}{\psi(a)}.
	$$	
	Integrating again results in $\log\psi(x)=O(x)$. This is a weaker growth restriction than \eqref{hyper-order}, and is satisfied by $\psi(x)=e^x$. Hence \eqref{hyper-order} needs to be assumed in Lemma~\ref{log-concave-lemma} even if $\psi$ would be twice differentiable. 
	
	(b) Let $\psi$ be the convex function in Lemma~\ref{not-ASV-lemma}, and let $C>0$ be a large constant. Then there exists an $N\in\N$ such that $x_n+C<x_{n+1}$ for all $n\geq N$. From \eqref{psi-on-the-interval2},
	\begin{equation}\label{log-derivative-bounded-below}
	\frac{\psi'(x)}{\psi(x)}\geq\frac{x_{n+1}^p-x_n^p}{\psi(x_n+C)(x_{n+1}-x_n)}\sim
	\frac{1}{C},\quad x_n<x<x_{n}+C,
	\end{equation}
	for all $n$ large enough. This means that the logarithmic derivative $\psi'/\psi$ does not tend to zero as $x\to\infty$. Further, let $y_n=\frac{x_n+x_{n+1}}{2}$. Then, from \eqref{psi-on-the-interval2},
	$$
	\frac{\psi'(y_n)}{\psi(y_n)}=\frac{2(x_{n+1}^p-x_n^p)}{(x_{n+1}^p+x_n^p)(x_{n+1}-x_n)}\sim \frac{2}{x_{n+1}}\to 0,\quad n\to\infty.
	$$
	Combining this with \eqref{log-derivative-bounded-below}, we find that $\psi'(x)/\psi(x)$ is not a decreasing function, which means that $\psi$ is not log-concave. In particular, differentiability (outside of a discrete set) together with a growth restriction such as \eqref{hyper-order} or \eqref{growth-psi} does not imply log-concavity. Further, the definition of the function $\psi$ can be slightly modified by smoothing the vertices in the graph of $\psi$ in such a way that $\psi$ becomes differentiable on $(x_1,\infty)$. This procedure has no affect on the properties of $\psi$ described in this remark. 
\end{remark}

Finally, having seen that a differentiable function $\psi$ satisfying the growth restriction \eqref{hyper-order} is not always log-concave, and its logarithmic derivative $\psi'/\psi$ does not always tend to zero as $x\to\infty$, it is natural to ask what is the size of the set
	\begin{equation}\label{E-epsilon}
	E:=\left\{ x>r_0: \frac{\psi'(x)}{\psi(x)}	
	>\varepsilon\right\},
	\end{equation}
where $\varepsilon>0$ is a fixed small constant? The following lemma answers this question by means of the continuous asymptotic density, that is, $\psi(x)=x$ in \eqref{new3}. 

\begin{lemma}
Let $\psi:(r_0,\infty)\to(1,\infty)$ be a strictly increasing, differentiable and unbounded function satisfying \eqref{hyper-order}, and let $\varepsilon>0$ be an arbitrarily small constant. Then the set $E$ in \eqref{E-epsilon} has an asymptotic density $\operatorname{dens}(E)=0$.
\end{lemma}

\begin{proof}
	By \eqref{hyper-order} and \eqref{E-epsilon},
	$$
	\int_{E\cap (r_0,x)}dt\leq \frac{1}{\varepsilon}\int_{r_0}^x\frac{\psi'(t)}{\psi(t)}\, dt\leq \frac{1}{\varepsilon}\left(\log\psi(x)+O(1)\right)=o(x),
	$$
	and the assertion follows.
\end{proof}


\section{Analytic and Abel densities}\label{B}

Recalling that
	$$
	\int_1^\infty\frac{dx}{x^p}=\frac{1}{p-1},\quad p>1,
	$$
the \emph{lower and upper analytic densities} of a set $A\subset\N$ can be expressed, respectively, by
	$$
	\underline{\operatorname{d}}_a(A)=\liminf_{p\to 1^+}(p-1)\sum_{n=1}^\infty \frac{\chi_A(n)}{n^p}\quad\textnormal{and}
	\quad\overline{\operatorname{d}}_a(A)
	=\limsup_{p\to 1^+}(p-1)\sum_{n=1}^\infty \frac{\chi_A(n)}{n^p},
	$$
see Section~\ref{equivalence-sec}.
If $\underline{\operatorname{d}}_a(A)=\overline{\operatorname{d}}_a(A)$, then their common value $\operatorname{d}_a(A)$ is called the \emph{analytic density} of $A$ discussed in Section~\ref{equivalence-sec}. We have seen in Theorem~\ref{Tenenbaum-thm} that for any $A \subset \N$, the analytic density $\operatorname{d}_a(A)$ exists if and only if the logarithmic density $\operatorname{\delta}(A)$ exists.
The technique used in the proof of Theorem~\ref{Tenenbaum-thm}, see \cite[p.~418]{Tenenbaum}, can be modified to prove Lemma~\ref{analytic-lemma} below. The lemma shows that the existence of the logarithmic density of $A\subset\N$ implies the existence of the analytic density of $A$, thus giving an alternative proof of one direction of Theorem~\ref{Tenenbaum-thm}.

\begin{lemma}\label{analytic-lemma}
	If $A\subset\N$, then
	\begin{equation}\label{analytic-logarithmic}
	0 \leq \underline{\operatorname{\delta}}(A) \leq \underline{\operatorname{d}}_a(A) \leq \overline{\operatorname{d}}_a(A) \leq \overline{\operatorname{\delta}}(A) \leq 1.
	\end{equation}
\end{lemma}

\begin{proof}
	It suffices to prove that
	$$
	\underline{\operatorname{\delta}}(A) \leq \underline{\operatorname{d}}_a(A)
	\quad\textnormal{and}\quad	
	\overline{\operatorname{d}}_a(A) \leq \overline{\operatorname{\delta}}(A).
	$$
The first inequality is trivial if $\underline{\operatorname{\delta}}(A)=0$, so we suppose that $\underline{\operatorname{\delta}}(A)>0$.	The proof of the second inequality above does not depend on this restriction.

Let $\varepsilon\in (0,\underline{\operatorname{\delta}}(A))$. Then there exists a $t_0=t_0(\varepsilon)>0$ such that
	$$
	\underline{\operatorname{\delta}}(A)-\varepsilon\leq \frac{A_{\log }(t)}{\log t}\leq \overline{\operatorname{\delta}}(A)+\varepsilon,\quad t\geq t_0,
	$$
where
	\[
	A_{\log}(x) =
\begin{cases}
	\sum_{1\leq k\leq x}\frac{\chi_A(k)}{k}, & \text{if } x\geq 1, \\
	0, & \text{if } 0\leq x<1.
\end{cases}
	\]
For any $x\geq 0$ there exists a unique non-negative integer $N$ such that $N\leq x<N+1$. Then, by \eqref{asym},
	$$
	A_{\log}(x)=A_{\log}(N)\leq\sum_{k=1}^N\frac{1}{k}\sim \log N\leq \log x.
	$$
Using Riemann-Stieltjes integration and integration by parts on the step-function $A_{\log}(n)$, we obtain
	\begin{equation}\label{analytic-dens}
	\begin{split}
		\sum_{n=1}^\infty \frac{\chi_A(n)}{n^p}&=
		\sum_{n=1}^\infty \frac{\chi_A(n)}{n^{p-1}n}
		=\sum_{n=1}^\infty \frac{A_{\log}(n)-A_{\log}(n-1)}{n^{p-1}}
		=\int_1^\infty \frac{dA_{\log}(t)}{t^{p-1}}\\
		&=\left[\frac{A_{\log}(t)}{t^{p-1}}\right]_1^\infty+(p-1)\int_1^\infty\frac{A_{\log}(t)}{t^p}\, dt\\
		&=O(1)+(p-1)\int_1^\infty\frac{A_{\log}(t)}{\log t}\cdot \frac{\log t}{t^p}\, dt,\quad p>1,
	\end{split}	
	\end{equation}
see \cite[Theorem~7.9]{Apostol} and \cite[Theorem~7.11]{Apostol}. Thus, using
	\begin{equation*}
	\int_1^\infty \frac{\log t}{t^p}\, dt=\frac{1}{(p-1)^2},\quad p>1,
	\end{equation*}
we obtain
	$$
	\underline{\operatorname{\delta}}(A)-\varepsilon+O(1)(p-1)\leq (p-1)\sum_{n=1}^\infty \frac{\chi_A(n)}{n^p}\leq \overline{\operatorname{\delta}}(A)+\varepsilon+O(1)(p-1).
	$$
	Taking limit inferior as $p\to 1^+$ on the first inequality and limit superior as $p\to 1^+$ on the second inequality gives us
	$$
		\underline{\operatorname{\delta}}(A)-\varepsilon \leq \underline{\operatorname{d}}_a(A)
	\quad\textnormal{and}\quad	
	\overline{\operatorname{d}}_a(A) \leq \overline{\operatorname{\delta}}(A)+\varepsilon.
	$$
	The assertions follow by letting $\varepsilon\to 0^+$.
\end{proof}

The other direction of Theorem~\ref{Tenenbaum-thm} can be proved by using the following Karamata's theorem, see \cite[p.~418]{Tenenbaum} for details. Karamata's theorem will be needed later on for another purpose in this appendix.

\begin{theorem}[{\cite[p.~322]{Tenenbaum}}]\label{karamata}
	Let $A(t)$ be an increasing function such that the integral
	$$
	F(\sigma):=\int_0^{\infty} \mathrm{e}^{-\sigma t} dA(t)
	$$
	converges for all $\sigma>0$. Assume there exist two real numbers $c \geqslant 0$, $\omega>0$, such that
	$$
	F(\sigma)=\frac{c+o(1)}{\sigma^\omega}, \quad\sigma \rightarrow 0+ .
	$$
	Then we have
	$$
	A(x)=\frac{(c+o(1) )x^\omega}{\Gamma(\omega+1)}, \quad x \rightarrow+\infty.
	$$
\end{theorem}

Recalling that
	$$
	\sum_{n=1}^\infty x^n=\frac{x}{1-x},\quad 0<x<1,
	$$
the \emph{lower and upper Abel densities} of a set $A\subset\N$ are defined, respectively, by
	$$
	\underline{\operatorname{d}}_\text{Abel}(A)
	=\liminf_{x\to 1^-}(1-x)\sum_{n=1}^\infty x^n\chi_A(n)\quad\textnormal{and}\quad
\overline{\operatorname{d}}_\text{Abel}(A)
	=\limsup_{x\to 1^-}(1-x)\sum_{n=1}^\infty x^n\chi_A(n).
	$$
If $\underline{\operatorname{d}}_\text{Abel}(A)=\overline{\operatorname{d}}_\text{Abel}(A)$, then their common value $\operatorname{d}_\text{Abel}(A)$ is called the \emph{Abel density} of $A$ discussed in \cite{Son}. It might not be obvious at first glance, but the Abel density is closely related to the asymptotic (linear) density, as the following theorem illustrates.

\begin{theorem}	[{\cite[Theorem~4.3]{Son}}]\label{Abel-thm}
Let $A\subset\N$. Then the Abel density of $A$ exists if and only if the asymptotic density of $A$ exists. In this case, the two densities are equal.
\end{theorem}

The technique used in proving Lemma~\ref{analytic-lemma} can be used to obtain a short alternative proof for Lemma~\ref{Abel-lemma} below. The lemma shows that the existence of the asymptotic density of $A\subset \N$ implies the existence of the Abel density of $A$, thus proving one direction of Theorem~\ref{Abel-thm}.

\begin{lemma}\label{Abel-lemma}
	\cite[Theorem~4.2]{Son}
	If $A\subset\N$, then
	\begin{eqnarray}\label{analytic-Abel}
	0 \leq \underline{\operatorname{d}}(A) \leq \underline{\operatorname{d}}_\textnormal{Abel}(A) \leq \overline{\operatorname{d}}_\textnormal{Abel}(A) \leq \overline{\operatorname{d}}(A) \leq 1.
	\end{eqnarray}
\end{lemma}

\begin{proof}
	It suffices to prove that
	$$
	\underline{\operatorname{d}}(A) \leq \underline{\operatorname{d}}_\textnormal{Abel}(A)
	\quad\textnormal{and}\quad	
	\overline{\operatorname{d}}_\textnormal{Abel}(A) \leq \overline{\operatorname{d}}(A),
	$$
where we may again suppose that $\underline{\operatorname{d}}(A)>0$. The proof of the second inequality above does not depend on this restriction.
  
We begin by noting that
	$$
	\int_1^\infty tx^t\, dt=\frac{(1-\log x)x}{\log^2x},\quad 0<x<1.
	$$
This gives, in particular, 
	$$
	\lim_{t\to\infty}tx^t=0,\quad 0<x<1,
	$$
which can also be verified independently by L'Hospital's rule.
	Let $\varepsilon\in (0,\underline{\operatorname{d}}(A))$. Then there exists a $t_0=t_0(\varepsilon)>0$ such that
	$$
	\underline{\operatorname{d}}(A)-\varepsilon\leq \frac{A(t)}{t}\leq \overline{\operatorname{d}}(A)+\varepsilon,\quad t\geq t_0,
	$$
where 
\[
A(x) =
\begin{cases}
	\sum_{1\leq k\leq x}\chi_A(k), & \text{if } x\geq 1, \\
	0, & \text{if } 0\leq x<1.
\end{cases}
\]
It is clear that $A(x)\leq x$ for $x\geq 0$. Using Riemann-Stieltjes integration and integration by parts on the step-function $A(n)$, we obtain
	\begin{equation}\label{RS-Abel}
	\begin{split}
	\sum_{n=1}^\infty x^n\chi_A(n)&=
	\sum_{n=1}^\infty x^n\big(A(n)-A(n-1)\big)\\
	&=\int_1^\infty x^tdA(t)
	=\left[x^tA(t)\right]_1^\infty-\log x\int_1^\infty x^tA(t)\, dt\\
	&=O(1)+\log \frac{1}{x}\int_1^\infty tx^t\cdot\frac{A(t)}{t}\, dt,\quad 0<x<1.
	\end{split}	
	\end{equation}
	Thus
	$$
	\left(\underline{\operatorname{d}}(A)-\varepsilon\right)\frac{(1-\log x)x}{-\log x}+O(1)\leq \sum_{n=1}^\infty x^n\chi_A(n)\leq \left(\overline{\operatorname{d}}(A)+\varepsilon\right)\frac{(1-\log x)x}{-\log x} +O(1).
	$$
	Since
	$$
	\lim_{x\to 1^-}(1-x)\frac{(1-\log x)x}{-\log x}=1,
	$$
	we find that there exists a $\delta\in (0,1)$ such that
	$$
	\underline{\operatorname{d}}(A)-2\varepsilon\leq (1-x)\sum_{n=1}^\infty x^n\chi_A(n)\leq \overline{\operatorname{d}}(A)+2\varepsilon,\quad 1-\delta<x<1.
	$$
	Taking limit inferior as $x\to 1^-$ on the first inequality and limit superior as $x\to 1^-$ on the second inequality gives us
	$$
	\underline{\operatorname{d}}(A)-2\varepsilon \leq \underline{\operatorname{d}}_\textnormal{Abel}(A)
	\quad\textnormal{and}\quad	
	\overline{\operatorname{d}}_\textnormal{Abel}(A) \leq \overline{\operatorname{d}}(A)+2\varepsilon.
	$$
	The assertion follows by letting $\varepsilon\to 0^+$.
\end{proof} 

We give a new proof for the other direction of Theorem~\ref{Abel-thm} based on Karamata's theorem.

\begin{lemma}
	Let $A\subset\N$. If $\delta:=\operatorname{d}_\textnormal{Abel}(A)$ exists, then $\operatorname{d}(A)$ exists and $\operatorname{d}(A)=\delta$.
\end{lemma}

\begin{proof}
	By the assumption, we have
	\begin{equation}\label{Abel-asymptotic}
	(1-x)\sum_{n=1}^\infty x^n\chi_A(n)=\delta+o(1),\quad x\to 1^-.
	\end{equation}
Moreover, from \eqref{RS-Abel},
	$$
	\sum_{n=1}^\infty x^n\chi_A(n)=\int_0^\infty x^tdA(t)=\int_0^\infty e^{-t\log\frac{1}{x}}dA(t).
	$$
	Recall that 
	$$
	1-x<\log\frac1x<\frac{1-x}{x},\quad 0<x<1.
	$$
	Set $\sigma=\log\frac1x$, and define
	$$
	F(\sigma):=\int_0^\infty e^{-\sigma t}dA(t),\quad \sigma>0.
	$$
	Then, by \eqref{Abel-asymptotic}, 
	$$
	F(\sigma)=\frac{\delta+o(1)}{\sigma},\quad \sigma\to 0^+.
	$$
	By Theorem~\ref{karamata}, we get
	$$
	A(n)=(\delta+o(1))n,\quad n\to\infty.
	$$
	This completes the proof.
\end{proof}

Finally, we note that relations between upper and lower asymptotic and Abel densities can be found in \cite{FJS}.

\bigskip
E-mail: \texttt{janne.heittokangas@uef.fi}

\textsc{University of Eastern Finland, Department of Physics and Mathematics, P.O.~Box 111, 80100 Joensuu, Finland}

\medskip
E-mail: \texttt{z.latreuch@nhsm.edu.dz}

\textsc{ National Higher School of Mathematics, Scientific and Technology Hub of Sidi Abdellah, P.O. ~Box 75, Algiers 16093, Algeria}


\begin{thebibliography}{99}
	
\bibitem{A}
    Alexander R.,
	\emph{Density and multiplicative structure of sets of integers}.
	Acta Arith.~\textbf{12} (1966/67), 321--332.
	
\bibitem{Apostol}
	Apostol T.~M.,
	\emph{Mathematical Analysis}.
	Second edition. Addison-Wesley Publishing Co., 
	Reading, Mass.-London-Don Mills, Ont., 1974.	

\bibitem{B}
	Barry P.~D.,
	\emph{The minimum modulus of small integral and subharmonic functions}.
	Proc.~Lond. Math. Soc.~(3) \textbf{12} (1962), 445--495.
	
\bibitem{C}
	Cargo G.~T.,
	\emph{Comparable means and generalized convexity}.
	J.~Math.~Anal.~Appl.~\textbf{12} (1965), 387--392.

\bibitem{FJS}
	Filip F., A.~Jankov and J.~\v{S}ustek,
	\emph{On relation between asymptotic and Abel densities.}
	J.~Number Theory \textbf{209} (2020), 451--466. 

\bibitem{Grekos}
	Grekos G., 
	\emph{On various definitions of density (survey).} 
	Tatra Mt.~Math.~Publ.~\textbf{31} (2005), 17--27.
		
\bibitem{HR}
	Halberstam H.~and K.~F.~Roth,
	\emph{Sequences}.
	Springer-Verlag,
	New York-Heidelberg-Berlin, 1983.		
			
\bibitem{JO}
	Jankovi\'c S.~and T.~Ostrogorski, 
	\emph{Convex additively slowly varying functions.}
	J.~Math.~Anal. Appl.~\textbf{274} (2002), no.~1, 228--238.
	
\bibitem{MS} 
	Mass\'e B.~and D.~Schneider,
	\emph{A survey on weighted densities and their connection with the first digit phenomenon}. 
	Rocky Mt.~J.~Math.~\textbf{41} (2011), no.~5, 1395--1415.
	
\bibitem{NZM} 
	Niven I., H.~S.~Zuckerman and H.~L.~Montgomery, 
	\emph{An Introduction to the Theory of Numbers}. Fifth edition.
	John Wiley \& Sons, Inc., New York, 1991.	
	
\bibitem{P} Pekara G.~C., 
	\emph{The Asymptotic Density of Certain Integer	Sequences}. 
	PhD thesis, Oklahoma State University, 1972.
		
		
\bibitem{PS} 
	Powell B.~J.~and T.~\v{S}al\'at,
	\emph{Convergence of subseries of the harmonic series and asymptotic densities of sets of positive integers}. 
	Publ.~Inst.~Math., Nouv.~Sér.~\textbf{50}(64) (1991), 60--70.
	
\bibitem{R}
	Rajagopal C.~T.,
	\emph{Some limit theorems.}
	Amer.~J.~Math.~\textbf{70} (1948), 157--166.
	
\bibitem{RV}
	Rohrbach H.~and B.~Volkmann,
  	\emph{Verallgemeinerte asymptotische dichten}.
  	J.~Reine Angew.~Math. \textbf{194 }(1955), 195--209.
	
\bibitem{Son}
	Sonnenschein D., 
	\emph{A General Theory of Asymptotic Density.}
	PhD thesis, University of British Columbia, 1971.

\bibitem{Taskovic}
	Taskovi\'c M.~R.,
	\emph{Survey on translational regularly varying functions}.
	Mathematica \textbf{7} (2003), 153--174.
		
\bibitem{Taylor}
	Taylor A.~E.,
	\emph{L'Hospital's rule.}
	Amer.~Math.~Monthly \textbf{59} (1952), 20--24. 
	
\bibitem{Tenenbaum}
	Tenenbaum G.,
	\emph{Introduction to Analytic and Probabilistic Number Theory.} 
	Translated from the second French edition (1995) by C.~B.~Thomas. 
	Cambridge Studies in Advanced Mathematics, 46. Cambridge University Press, Cambridge, 1995.
	
\end{thebibliography}
\end{document}